\documentclass{imsart}

\usepackage{stackengine}

\usepackage[textsize=footnotesize]{todonotes}
\usepackage{xcolor}
\usepackage{hyperref}

\usepackage{mathrsfs}
\usepackage{graphicx}
\usepackage{amsmath}        
\usepackage{amssymb}
\usepackage{amsthm}
\usepackage{bm}
\usepackage[mathscr]{euscript}
   
        \usepackage[countmax]{subfloat}  
\usepackage{fourier}
\usepackage{xcolor}
\usepackage{mathtools}
\usepackage{enumerate}
\usepackage{enumitem}
\counterwithin{figure}{section}
\usepackage[numbers]{natbib}
\usepackage{url}
\usepackage{verbatim}
\usepackage{booktabs}
\usepackage{placeins}
\usepackage{multirow}
\usepackage{subcaption}
\usepackage[cmtip,all]{xy}
\usepackage{aliascnt}
\usepackage{ifthen}
\hypersetup
{
    pdfauthor={Anne van Delft},
    pdfkeywords={functional data, time series, spectral analysis,martingales}
}
\newcounter{alphcount}
{\begin{list}{{\upshape(}\alph{alphcount}\/{\upshape)\ }}%
             {\usecounter{alphcount}\labelwidth1.5em%
              \bigmargin2em\labelsep0.5em\topsep0.25em plus 0.5ex%
              \itemsep0.25em plus 0.5ex\parsep0em}}{\end{list}}
{\begin{list}{{\upshape(#1\arabic{alphcount})\hfill}}%
             {\usecounter{alphcount}\labelwidth2.5em%
              \bigmargin2.5em\labelsep0em\topsep0.25em plus 0.5ex%
              \itemsep0.25em plus 0.5ex\parsep0em}}{\end{list}}
\newcommand\tageq{\addtocounter{equation}{1}\tag{\theequation}}

\newcommand{\mynewtheorem}[2]{
  \newaliascnt{#1}{dummy}
  \newtheorem{#1}[#1]{#2}
  \aliascntresetthe{#1}
  \expandafter\def\csname #1autorefname\endcsname{#2}
}
 
\theoremstyle{plain}
  \mynewtheorem{thm}{Theorem}
  \mynewtheorem{Proposition}{Proposition}
  \mynewtheorem{Corollary}{Corollary}
  \mynewtheorem{assumption}{Assumption}
  \mynewtheorem{lemma}{Lemma}
\theoremstyle{definition}
  \mynewtheorem{definition}{Definition}
  \mynewtheorem{example}{Example}
\theoremstyle{remark}
  \mynewtheorem{Remark}{Remark}

\newcommand{\im}{\mathrm{i}}

\newcommand{\mb}[1]{\mathbb{#1}}
\newcommand{\mc}[1]{\mathcal{#1}}

\newcommand{\znum}{\mathbb{Z}}
\newcommand{\cnum}{\mathbb{C}}
\newcommand{\rnum}{\mathbb{R}}
\newcommand{\nnum}{\mathbb{N}}
\newcommand{\E}{\mathbb{E}}
\newcommand{\pr}{       \mathbb{P}}

\newcommand{\flo}[1]{\lfloor#1\rfloor}
\newcommand{\F}{\mathcal{F}}
\newcommand{\G}{\mathcal{G}}

\newcommand{\bv}{\Big\vert}
\newcommand{\hti}{\mathfrak{h}}
\newcommand{\inprod}[2]{\langle #1,#2 \rangle}
\allowdisplaybreaks

\newcommand{\rkhs}{\mc{H}}
\newcommand{\bignorm}[1]{\Big{\|}{#1}\Big{\|}}

\newcommand{\biginprod}[2]{\Big\langle #1, #2 \Big\rangle}

\newcommand{\longsquiggly}{\xymatrix@C=1.5em{{}\ar@{~>}[r]&{}}}
\newcommand{\medsquiggly}{\xymatrix@C=1.2em{{}\ar@{~>}[r]&{}}}

\providecommand{\AMS}[1]{\textbf{\textit{AMS subject classification: }} #1}

\newcounter{relctr} 
\everydisplay\expandafter{\the\everydisplay\setcounter{relctr}{0}} 

\AtBeginDocument{}

\newboolean{withproof}
\setboolean{withproof}{true}

\begin{document}

\begin{frontmatter}
    
\title{
Characterizing metric-space-valued processes: separating classes and weak invariance principles for measure-theoretic inference}

\maketitle

\begin{aug}
\author{\fnms{Anne} \snm{van Delft}\ead[label=e1,mark]{anne.vandelft@columbia.edu}}

\address{Department of Statistics, Columbia University, 1255 Amsterdam Avenue, New York, NY 10027, USA.\\}
\end{aug}

\begin{abstract}

This article investigates stochastic processes taking values in metric spaces that lack a topological vector space structure, a regime characterized by intricate interplay between topological, geometric, and temporal dependence structures. It is formally established that spaces admitting an isometric Hilbertian embedding constitute a strict subclass within the much broader class of metric spaces possessing the ball property. While traditional kernel methods are susceptible to geometric distortion when the underlying space cannot be isometrically embedded into a Hilbert space, we bypass such limitations by exploiting a fundamental structural property inherent to this broader class; namely, that Borel probability measures are uniquely determined by their values on balls. These separating classes provide the foundation for the subsequently introduced measure-theoretic inference methodology. We derive uniform convergence of a family of time-dependent random measures, alongside weak invariance principles for the corresponding nonstationary random fields. This framework explicitly exposes how dependence and geometric complexity influence sample path regularity. Furthermore, because the rapid decay of small-ball probabilities can prohibit the existence of limiting distributions for supremum-based discrepancy measures, we develop $L^p$-based alternatives. By directly leveraging the introduced convergence results, this approach circumvents the need for higher-order $U$-process formulations. Finally, for spaces that do admit an isometric Hilbertian embedding, and where $U$-processes naturally arise, we establish limit theory for both degenerate and nondegenerate multi-parameter $U$-processes, and demonstrate that local discrepancy tests maintain asymptotic stability under dynamic parameter regimes.
 \end{abstract}

\AMS{Primary 60F17, 60B05, 60G57; Secondary 62M10, 62G20, 62E20, 60G15}.


\end{frontmatter}


\section{Introduction}

The complex mathematical structure of modern data sets requires inference techniques that account not just for inherent spatial and temporal dynamics, but also for geometric and topological structure. To accommodate this, it is important to view data as points in a metric space that lacks a topological vector space structure. The incompatibility with a vector space structure means only the points and a notion of distance between them are available. Inference under this restriction effectively forces the specification of a class of functions that is rich enough to characterize probability measures, where the intrinsic geometry plays a key role in understanding how `rich' this class should be.

Various metrics may induce the same topology and thus the same Borel $\sigma$-algebra, but some metrics are more compatible with the intrinsic geometry of the space. For example, consider $\mb{S}^1=\{ (x_1,x_2) \in \rnum^2: x_1^2+x_2^2=1\}$, the unit circle. We may view this as a metric subspace of $\rnum^2$ endowed with the standard Euclidean metric. However, the  shortest arc distance (geodesic metric) is more compatible with the intrinsic geometry of the space. Indeed, even though both metrics induce the same topology when restricted to the circle and hence the same open sets, the Euclidean distance measures closeness of two points using points that are not on the circle, failing to respect the intrinsic geodesic structure. More generally, if a chosen embedding is not isometric, metric distortions can obscure local geometric features, making it difficult to distinguish differences between probability measures.

The aim of this article is two-fold. The first goal is to establish tractable separating classes for measure-theoretic inference that respect the intrinsic geometry. In settings where the space is isometric to a subset of a Hilbert space, the provided characterization of the separating class naturally places inference in the realm of reproducing kernel Hilbert spaces (RKHS) via Schoenberg's theorem \cite{schoenberg1938}. However, if no such isometric embedding exists then a richer class of functions should be considered. In particular, we establish in this article that spaces that are isometrically embeddable into a Hilbert space form a strict subclass of the metric spaces for which Borel measures are determined by their values on balls \cite{Christensen,PreissTisier91}---a characteristic we henceforth define as the \textit{ball property}. Crucially, because this determination of measures does not automatically pull back to arbitrary metric subspaces, this containment is not an immediate consequence of the property holding on the ambient Hilbert space. To fully characterize this demarcation, a necessary and sufficient condition is provided for the ball property to hold in general metric spaces, a result of independent interest. A key insight is that although a valid Hilbertian embedding may fail due to geometric constraints, this broader framework allows us to preserve statistical distinguishability while fully respecting the intrinsic geometry.

Evaluation of processes taking values in spaces that possess the ball property on the separating classes naturally yields multi-parameter nonstationary random fields. The second goal is to derive weak invariance principles and limit theory for these fields, establishing the necessary foundations for measure-theoretic inference.

This article is organized as follows. \autoref{sec:sec2} establishes the geometric framework and characterizes separating classes for spaces that possess the ball property, including a detailed investigation into reproducing kernel Hilbert spaces on separable metric spaces and isometric embeddings into a Hilbert space. Crucially, we establish an impossibility result demonstrating that global supremum-norm density (often called universality) of an RKHS cannot exist on any non-compact Polish space, highlighting the fundamental topological boundaries of standard kernel embedding frameworks.

\autoref{sec:sec3} introduces the mathematical foundations for measure-theoretic inference. Within this section, invariance principles for nonstationary random fields are derived under general metric geometries as well as under Hilbertian geometry. This is done for the class of locally stationary Polish-valued processes as defined in \cite{vDB25}. Limiting representations of the partial sum processes are made explicit, providing a meaningful basis for testing for validity of model constraints. The focus is first on the larger classes that possess the ball property. We investigate the conditions under which supremum-based discrepancy measures in the form of integral probability metrics (such as the Wasserstein-1 distance, which is fundamental in optimal transport due to the Kantorovich-Rubinstein duality theorem) can be reasonably estimated. We also address scenarios where the intrinsic dimension forces the consideration of $L^p$-based discrepancy measures. For the latter, limit theory is established by directly combining the uniform convergence result of the empirical random measures with the weak invariance principles for the random fields. 
It will become clear that the geometric complexity plays a role in the conditions necessary to ensure sample path regularity for the sequential estimators used to construct these discrepancy measures. 

In the subsequent section, weak invariance principles are introduced for multi-parameter nonstationary $U$-processes. Characterizations of measures or corresponding test statistics can often be cast in the form of a continuous transformation of a multi-parameter $U$-process with a kernel function that satisfies certain properties. This structural representation naturally encompasses the RKHS framework.  Limit theorems are given when the $U$-process is  degenerate as well as when it is non-degenerate. The invariance principle in the nondegenerate case does not rely on strict positive-definiteness of the kernel.  The weak invariance principles are showcased to construct tests for local discrepancies in measure.  Notably, the weak invariance principles introduced here serve as a robust foundation for measure-theoretic inference across a wide spectrum of assumptions, ranging from iid Polish-valued processes to complex nonstationary processes of random geometric objects.
   
\section{Geometric framework and separating classes}\label{sec:sec2}
Consider a measurable space $(S,\mc{S})$. Throughout this article, we assume that $S$ is a separable metrizable topological space and that $\mc{S}=\mc{B}(S)$ is the Borel $\sigma$-algebra on $S$.  
 $\Gamma$ will denote a class of measurable functions $f: S \to \rnum$, though sometimes we consider $f: S \to \cnum$ to incorporate characteristic functionals.

Let $\mc{P}(S)$ be the space of Borel probability measures on $S$. We recall that a collection $\Gamma$ separates measures on $(S, \mc{S})$ whenever, for any  $\mu, \nu \in \mc{P}(S)$, 
    \[
    \int f(x) \mu(dx) = \int f(x) \nu (dx) \quad \forall f \in \Gamma
    \]
  implies that 
  $\mu=\nu$. For example, if $S=\rnum$, then $\Gamma=\{e^{\im \lambda (\cdot)}: \lambda \in \rnum\}$, where $\im=\sqrt{-1}$, is a separating class \cite{Levy1925}. For multivariate distributions the 1-dimensional projections are separating \cite{CramerWold1936}. More generally, Kolmogorov showed that the finite-dimensional distributions of a probability measure on a function space form a separating class
       \citep[see e.g.,][]{kolmogorov1933,bil68}. Parzen \cite{parzen1963} made use of this result to show that the mean functional of a Gaussian process with fixed covariance kernel uniquely identifies the law of the process if and only if the kernel is strictly positive definite.

 Suppose we are interested in testing whether any two Borel probability measures $\mu, \nu$ on $S$ are equal. In view of the above, a common and obvious choice to formalize a test is to consider what is now known as \textit{the integral probability semi-metric}, a term coined by \cite{Muller1997}. Defined by the class $\Gamma$, this semi-metric is given by
    \[
    \gamma_{\Gamma}(\mu,\nu):= \sup_{f \in \Gamma}\bv \int f  d\mu - \int f  d\nu\bv.\tageq \label{eq:ipsm}
    \]
It is a metric in case $\gamma_{\Gamma}(\mu,\nu)=0$ implies $\mu=\nu$. 
Note that this inherently means that the class of functions $\Gamma$ must be rich enough to form a separating class for the space of Borel probability measures on $S$. Such supremum-based metrics over function classes were first considered by \cite{fortet1953}, and later by \cite{Prokhorov1956,Dudley1966} to study and  metrize convergence of probability measures. 
Note that the total variation distance, the Dudley metric and the Kolmogorov metric are all particular cases  of \eqref{eq:ipsm} induced by different choices for the class $\Gamma$. Noteworthy is also the case where the class $\Gamma$ is taken to be the unit ball in a reproducing kernel Hilbert space. This was pioneered by Parzen \cite{parzen1963}  who effectively showed that its square corresponds to the squared distance of the mean embeddings. In the 1980's, the corresponding distance became known as \textit{Energy distance} \cite{Szekely1985}. More recently, it has become popularized in the machine learning community under the name \textit{maximum mean discrepancy} \cite{gretton2012,borgwardt2006,sriperumbudur2010}.

The choice of $\Gamma$ and hence the particular (closed) form of \eqref{eq:ipsm} will play a role in the properties of this semi-metric. We are interested in the scenario where the only information the researcher has are the points in the space and a notion of distance between the points, for which the researcher must (a priori) specify a metric. We can often equip $S$ with various metrics, say $d_1$ and $d_2$ that will induce the same topology $\tau_S$ on $S$. Consequently, the identity map  $i: (S,d_1) \to (S, d_2)$ is a homeomorphism. An example was given in the introduction for the circle $\mb{S}^1$. Given that both metric topologies generate the same open sets, it becomes natural to consider suitable functionals of the distances induced by one of these metrics for the function class. However, choosing a metric that is not compatible with the intrinsic geometry could lead to invalid inference. 
For example, the topological structure of the circle is non-contractible and is therefore incompatible with any topological vector space structure. Spaces that are compact or have a boundary will also be incompatible as these lack 
 the required properties to support linear operations, but as long as its intrinsic curvature is flat enough and the space is contractible (e.g., the disc or unit square) then such spaces can still be isometrically embedded into a Hilbert space. This embeddability is an extremely attractive property from the point of view of inference due to Schoenberg's condition\cite{schoenberg1938}. Essentially, in this case, a natural choice for the class of $\Gamma$ is the unit ball in a reproducing kernel Hilbert space which reduces to the estimation of $U$-statistics. 
However, spaces that are non-contractible or are contractible but exhibit positive curvature such as $\mb{S}^n$ will not be isometric with a subset of a Hilbert space, making these types of spaces not compatible with (geometry-aware) inference via an RKHS. Furthermore, negatively curved spaces may also fail to be compatible \cite{Davies}.

\subsection{Characterization of separating classes for spaces that possess the ball property}

As a starting point, denote by $B_r(x) = \{y \in S: d(x,y) < r\}$ the open ball centered at $x$ of radius $r$, and ${\bf B}=\{B_r(x): x \in S, r >0\}$ the collection of open balls. The following provides a necessary and sufficient condition for Borel measures to be fully characterized by their values on balls. 
  \begin{lemma}\label{lemm:ballsspan}
    Let $(S, d_{S})$ be a separable metric space. Then the linear span of $\{1_B: B \in {\bf B}\}$ is dense in $L^1(S,\mc{S},\mu)$ for every Borel probability measure $\mu$ if and only if 
    ${\bf B}$ forms a separating class. 
\end{lemma}

\begin{proof}[Proof of \autoref{lemm:ballsspan}]
    Denote the linear span of the indicators of the open balls by $\mc{A}=\text{span}\{1_B: B \in {\bf B}\}$ and assume that it is dense for all Borel probability measures $\mu$ on $S$. Suppose 
    \[
    \mu_1(B)=\mu_2(B) \quad\forall B \in {\bf B}.
    \]
    then for all $f \in \mc{A}$, 
    \[
    \int f d\mu_1 = \int f d\mu_2. 
    \]
    Define $\mu= \frac{1}{2} \mu_1 + \frac{1}{2}\mu_2$. Then, by the Radon-Nikodym theorem, there exist $g_1, g_2 \in L^\infty(\mu)$  (since $\mu_i(E) \le 2\mu(E)$) such that $\mu_i= g_i d\mu$ for $i=1,2$, and 
    \[
    \int f g_1 d\mu -\int f g_2  d\mu =0 \quad \forall f \in \mc{A} \implies \int f (g_1-g_2) d\mu =0\quad \forall f \in \mc{A}.
    \]
  The denseness of  $\mc{A}$  in $L^1(S,\mc{S},\mu)$ implies $g_1-g_2=0$ $\mu$-almost everywhere, which in turn implies that $\mu_1=\mu_2$. Hence, ${\bf B}$ is a separating class. 
    Conversely, assume ${\bf B}$ is a separating class. Recall that it follows from the Hahn-Banach separation theorem that $\mc{A}$ is dense in $L^1(S,\mc{S},\mu)$ if and only if the only continuous linear functional  that vanishes on $\mc{A}$ is the zero functional. Consider an  element in the topological dual space $g \in L^\infty$ such that 
    \[
    \int 1_B(x) g(x) \mu(dx) =0 \forall B \in {\bf B}
    \]
    Then $\eta:=g d\mu$ is a signed measure. We write its Jordan decomposition $\eta = \eta^+-\eta^-$ and the above equation yields that $\eta^+(B)=\eta^{-}(B)$ for all $B \in \bf{B}$ but as $\bf{B}$ is a separating class by assumption this means that $\eta=0$, which implies that $g=0$ $\mu$-almost everywhere.   
\end{proof}

This has important implications for understanding the interplay between the geometry of a space and the determination of measures. Namely, 
any metric space for which ${\bf B}$ is a separating class has the property that any Borel probability measure on it is determined by its values on balls. 
This ball property holds true for a wide variety of metric spaces such as compact Riemannian manifolds and separable Banach spaces \cite{PreissTisier91} and thus for separable Hilbert spaces.

For metric spaces that do not possess the ball property, the indicator of the intersection of the balls fails to be in the closure of the linear span. These intersections must therefore contain extra structural information that cannot be reconstructed or approximated just by knowing the measure on the balls separately. In this case, we can change the measure on the intersection without changing the measure on the individual balls. A famous example of a  compact metric space for which the ball property fails to hold was constructed in \cite{Davies}.

By Kuratowski's embedding, any separable metric space (and thus any compact metric space) can be isometrically embedded into a subset of a separable Banach space. It is important to emphasize that, even though Borel measures on the separable Banach space possess the ball property, this does not imply that the measures on the original space are also determined by their values on balls. 
\begin{Proposition}
    The property that measures are determined by their values on balls is not preserved under isometric embeddings into subsets of Banach spaces. That is, if the ambient space has the ball property, then this property does not pull back to the original metric space.
\end{Proposition}

Obviously, Davies' counterexample establishes this point; a direct geometric argument is also instructive. Let $(X,d_X)$ be a separable metric space and let $(Y,d_Y)$ be a separable Banach space. Let $\phi: X \to Y$ be an isometric embedding. If we let $Y^\prime= \phi(X)$, then we can easily see that
    \[
    B_{Y^\prime}(\phi(x),r) = B_Y(\phi(x),r) \,\cap\, \phi(X).
    \]
Consequently, unless $\phi(X)=Y$, the 'blindness' of the balls is not fixed by such an isometric embedding. Inherently, conducting inference in the ambient space via balls by no means justifies conclusions regarding the measures on the original space. However, we are in more luck if the metric space is isometric to a subset of a Hilbert space, in which case simpler separating classes become available. To demonstrate how this special Hilbertian case fits into the broader theory, we present the following  characterization of separating classes for metric spaces that satisfy the ball property.

\begin{lemma}\label{lem:cdfslapchar}
    Let $(S,d_{S})$ be a separable metric space and let $\mc{S}$ denote its Borel $\sigma$-algebra. The following are equivalent:
    \begin{enumerate}
        \item  Every Borel probability measure $\nu$ on $S$ is determined by its values on balls. 
        \item The collection 
    $\Gamma_{\varphi}:=\{f_{y,r}(x) = e^{-r d_{S}(x,y)}: y \in S, r >0\}$ forms a separating class. 
    \item The collection $\Gamma_\phi: =\{f_{y,\lambda}(x) = e^{\im \lambda d_{S}(x,y)} : y \in S, \lambda \in \rnum\}$, where $\im=\sqrt{-1}$, forms a separating class.
    \end{enumerate}
\end{lemma}

{
\begin{proof}[Proof of \autoref{lem:cdfslapchar}]
Straightforward from the one-to-one correspondences between measures on $\rnum$, their characteristic functionals, and Laplace transforms for nonnegative random variables.
\end{proof}

We will show that any separable metric space for which a kernel 
from the Schoenberg class, such as the Laplace or Gaussian kernel, is strictly positive-definite will satisfy the ball property (see \autoref{thm:isomdeterminedbyballs}). 
Taking in \eqref{eq:ipsm} the class $\Gamma$ to be the unit ball in the reproducing kernel Hilbert space $\rkhs_k$ then provides a true metric and takes on a particularly succinct closed form.

\subsubsection{Reproducing kernel Hilbert spaces on separable metric  spaces}

The importance of strictly positive definite kernels as well as the consideration of RKHSs was  solidified and exploited in the probabilistic, time series and harmonic analysis literature  during the previous century \cite{bochner1932,aronszajn1950,parzen1962,parzen1963,berg1984}. 
We shall make use of the following result \cite{aronszajn1950}:
\begin{thm}
Let $\mc{X}$ be a non-empty set and assume $k: \mc{X} \times \mc{X}$ is a positive-definite kernel. Then there exists a unique Hilbert space, denoted by $\rkhs_k$, such that 
\begin{enumerate}[label=\roman*)]
 \item $\rkhs_k$ consists of real-valued functions on $\mc{X}$;
 \item The family of functions $\{K_x: x\in \mc{X}\}$, where $K_x(y) = k(x,y)$ is contained in $\rkhs_k$, i.e., $\{K_x: x\in \mc{X}\} \subset \rkhs_k$;
\item For all $x \in \mc{X}$ and $f \in \rkhs$ the so-called \textit{reproducing property} holds: \[\inprod{f}{K_x}_{\rkhs_k} = f(x). 
\]
\end{enumerate}
\end{thm}

More can be said if $\mc{X}$ has a topology defined on it. Particularly, we will focus on the following refinement of above statement \citep[see e.g.,][]{vakhania1987}. 

\begin{Proposition}\label{prop:contemb} 
Let $(\mc{X},d_{\mc{X}})$ be a separable metric space, where $\mc{X}$ is equipped with its Borel $\sigma$-algebra. Assume that 
$k: \mc{X} \times \mc{X} \to \rnum$ is continuous. Then in addition to the above properties, the RKHS $\rkhs:=\rkhs_k$ consists of continuous functions and is a separable Hilbert space. Furthermore, if the kernel is bounded then the embedding $\rkhs \hookrightarrow (C_b(\mc{X}), \|\cdot\|_\infty)$, where $C_b(\mc{X})$ denotes the space of bounded continuous functions on $\mc{X}$, is continuous.
\end{Proposition}

The evaluation functional $\pi_x:f \mapsto f(x)$ is an element of the topological dual of $\rkhs$ for every $x \in \mc{X}$. By the Riesz-Fr{\'e}chet theorem, there exists a unique element $K_x \in \rkhs$ such that 
\[
\pi_x(f) = \inprod{f}{K_x}_{\rkhs}.
\]
Furthermore, taking $f=K_y$ gives the well-known identity $\pi_x(K_y) = \inprod{K_x}{K_y}_{\rkhs}=k(x,y)$. 
For a given probability measure $\mu$ on $\mc{B}(\mc{X})$, define the integration functional 
   $ I_\mu : \rkhs \to \rnum$, $f \mapsto  \int_{\mc{X}} f(x)\mu(dx)$, 
which is an  element of the topological dual space $\rkhs^\prime$ under the conditions of the above proposition. Hence, there exists a unique element $m_{\mu} \in \rkhs$, known as the \textit{mean functional of $\mu$}, such that for every $f \in \rkhs$, $
 I_\mu(f)=\inprod{f}{m_{\mu}}$.
The reader can immediately verify that \[
m_{\mu}=\int_{\mc{X}} k(\cdot,y) \mu(dy),\tageq \label{eq:pettis}\] 
which is generally also known as the Pettis integral \cite{pettis1938}, and which coincides with the Bochner integral under the above conditions. We remark that minimal conditions necessary to ensure well-definedness of the Pettis integral are that the kernel is Borel-measurable and is square root-integrable with respect to the specified measure. For the purpose of this paper, it is without loss of generality to consider continuous kernels. 

Parzen referred to \eqref{eq:pettis} as the \textit{mean functional} and pioneered the application of  RKHS in probability and statistics, particularly to define probability density functionals and to uniquely identify probability measures via the injectivity of the mapping \cite{parzen1962,parzen1963}. It is crucial to recognize that what is currently called \textit{mean embedding} in the machine learning literature was  rigorously established by Parzen as the mean functional. Indeed, Parzen's work predates the literature on "characteristic kernels" by over four decades. The requirement that a kernel generates an RKHS that is dense in a convergence-determining class such as $(C_0(\mc{X}), \|\cdot\|_\infty)$ where $\mc{X}$ is a locally compact metric space, has been revived under the label of universality in the machine learning literature for specific kernels on non-compact domains (see e.g., \cite{gretton2012,sriperumbudur2010,simongab2023}).  
 Generally, however, the metric space is too complex to support an RKHS that induces a convergence-determining class under the supremum norm, a limitation formalized by the following result.
\begin{Proposition}\label{prop:claim}
    If $(\mc{X},d_{\mc{X}})$ is non-compact but Polish, then no RKHS can exist that is dense in $(C_b(\mc{X}),\|\cdot\|_{\infty})$.
\end{Proposition}

\begin{proof} Assume first that the RKHS is separable. In a metric space, the closure of a separable set is separable. However, $(C_b(\mc{X}),\|\cdot\|_\infty)$ is separable if and only if $\mc{X}$ is compact. Hence, no separable RKHS on a non-compact Polish space can be dense in $(C_b(\mc{X}),\|\cdot\|_\infty)$. It remains to show that any RKHS satisfying $\rkhs \subset C_b(\mc{X})$ must be separable. The fact that $\rkhs \subset C_b(\mc{X})$ implies that for any $f \in \rkhs$, the set $\{\delta_x(f): x \in \mc{X}\} \subset \rnum$ is uniformly bounded by $\|f\|_\infty$. The uniform boundedness principle therefore implies that the kernel must be bounded. Furthermore, $\rkhs \subset C_b(\mc{X})$ implies that the functions $y\mapsto k(x,y)$ are continuous. Since $\mc{X}$ is Polish, it contains a countable dense subset, say $D$. Any $f \in \rkhs$ that evaluates to zero on this dense subset must be identically zero due to continuity. Therefore, the orthogonal complement of the set $\{k_x: x \in D\}$ is trivial.
The result now follows. \end{proof}

With these things in mind, let us revisit the integral probability semi-metric for $\Gamma=\{f \in \rkhs: \|f\|_{\rkhs} \le 1\}$. Consider two well-defined mean functionals $m_{\mu},  m_\nu \in \rkhs$ induced by the Borel probability measures $\mu$ and $\nu$, respectively. 
Standard calculations yield that the square of the integral probability (semi-)metric on the unit ball of  $\rkhs$ is simply given by 
\[
\gamma^2_{\mathrm{Ball}_{\rkhs}}(\mu,\nu)= \|m_{\mu}-m_{\nu}\|^2_{\rkhs} \tageq \label{eq:gamballrkhs}
\]

Hence, for $\gamma_{\Gamma}$ to be a well-defined metric on the space of Borel probability measures $\mc{P}(S)$, we require that 
\[
 \|m_{\mu}-m_{\nu}\|^2_{\rkhs} =0 \quad \implies \quad \mu =\nu.
\]
For this to hold, a continuous kernel must be strictly positive definite \citep{parzen1962,schwartz1966,berg1984},  
since
\[ \|m_{\mu}-m_{\nu}\|^2_{\rkhs} = \int \int k(x,y)\eta(dx) \eta(dy) \tageq \label{eq:injememb}\]
where $\eta=\mu-\nu$ is a signed Borel measure. 

Having a (strictly) positive definite kernel induced by the metric imposes a constraint on the metric geometry. This restricts in turn the applicability of RKHS-based inference on general metric spaces.

\subsubsection{
Isometric embeddings into a Hilbert space}

Almost a century ago, Schoenberg proved in his highly impactful works \cite{schoenberg1938,schoenberg1938monotone} the key results to establish that one cannot take positive definiteness for granted if one wants to avoid distorting the metric geometry. Schoenberg characterized all continuous positive-definite functions on a Hilbert space that only depend on the norm:
\begin{thm}\label{thm:schoenbergrep}
    Let $H$ be an infinite-dimensional Hilbert space, and let $f: \rnum_+ \to \cnum$ be a continuous function with $f(0)=1$. Then the functional $\chi(h) = f(\|h\|)$, $h \in H$, is positive definite if and only if it can be represented as
    \[
    f(t) = \int_0^\infty e^{-s\, t^2} d\nu(s) \tageq \label{eq:schoenbergke}
    \]
    where $\nu$ is a finite Borel measure on $[0,\infty)$.
\end{thm}

This gives rise to a class of continuous  kernels of the form $k(x,y) = f(\|x-y\|)$, where $f$ is as in \eqref{eq:schoenbergke}. In essence, Schoenberg's theorem implies that every such kernel is positive definite if and only if it is a scale mixture of Gaussians with respect to a non-negative finite Borel measure on $[0,\infty)$. This framework encompasses the Matérn kernel family, of which the Gaussian kernel is a foundational member, as well as the Laplace kernel $k(x,y)=e^{-\alpha\|x-y\|}$. These kernels satisfy $\text{supp}(\nu) \cap (0,\infty) \neq \emptyset$ and are therefore \textit{strictly} positive definite, ensuring that the injectivity condition holds on the space of Borel probability measures on $H$.

On a Hilbert space, an RKHS induced by a kernel in the Schoenberg class which satisfies $\text{supp}(\nu) \cap (0,\infty) \neq \emptyset$ is dense in $(C_0(H), \|\cdot\|_\infty)$ if $H$ is locally compact (i.e., finite-dimensional). If $H$ is infinite-dimensional, then the same density property holds true if we instead equip the space $C_0(H)$ with the compact-open topology. Furthermore, such Schoenberg kernels 
are dense in $L^1(H,\mc{B}(H),\mu)$ for any $\mu \in \mc{P}(H)$.

For metric spaces that can be isometrically embedded into a Hilbert space, this separating property carries over to strictly positive definite kernels defined by $k(x,y)=f(d_{\mc{X}}(x,y))$, $x, y \in \mc{X}$, where $f$ is as in \eqref{eq:schoenbergke}. Schoenberg identified these metric spaces via the properties of the metric:

\begin{thm}\label{thm:isombedSch}
    Let $(\mc{X},d_\mc{X})$ be a metric space. Then $(\mc{X},d_\mc{X})$ is isometric to a subset of a Hilbert space if and only if $d^2_{\mc{X}}$ is a conditionally negative definite function. 
\end{thm}

For any hermitian function $\tau(s,t)=\overline{\tau(t,s)}$, $s,t \in \mc{X}$, where $\mc{X}$ is non-empty, the function $e^{-\alpha \tau}$ is positive definite for any $\alpha>0$ if and only if the function $\tau$ is conditionally negative definite. Furthermore if the space $\mc{X}$ is a topological group, then a nonnegative function $(s,t) \mapsto \rho(s-t)$ that is conditionally negative definite, has the property that $\rho^\alpha$ is also conditionally negative definite for every real $0 \le \alpha \le 1$. Hence, the work of Schoenberg allows one to conclude that if a  metric space satisfies \autoref{thm:isombedSch}, then strict positive definiteness of Schoenberg kernels satisfying $\text{supp}(\nu) \cap (0,\infty) \neq \emptyset$,   is inherited.

Unfortunately, many interesting metric spaces are not isometric to a subset of a Hilbert space. In particular, this holds for any Banach space that is not a Hilbert space \cite{JordanNeumann1935}, in which case the representation \eqref{eq:schoenbergke} fails. More generally, any metric space that can be isometrically embedded into a Hilbert space must satisfy the parallelogram law, which structurally restricts it to flat geometries. In other words, spaces that have either intrinsic positive curvature (like spheres) or negative curvature (hyperbolic spaces) are not compatible with the rigid geometry of a Hilbert space. This geometric constraint directly affects inference methods that rely on embedding measures into Hilbert spaces. For example, while Gaussian embeddings into an RKHS elegantly exploit the dichotomy of Gaussian measures to establish a highly powerful ‘separation of measure’ phenomenon as recently shown in \cite{sanwagpan26}, such approaches remain inherently vulnerable to these metric distortions if \autoref{thm:isombedSch} fails to hold. 
Hypothesis tests for intrinsic geometric features such as zero-curvature as in \cite{songmuller26} could serve as a highly sophisticated diagnostic tool to determine whether a Hilbertian embedding is reasonable. However, such elegant tools are currently only applicable to geodesic spaces for which central dispersion measures such as Fr{\'e}chet variances are well-defined. 
Complementing these specialized geometric approaches, the framework put forward in this paper encompasses these settings by operating directly via the separating class property of balls. While the separating class property of balls also does not hold universally for all metric spaces (e.g., \cite{Davies}), the framework put forward in this paper sidesteps Hilbertian constraints or the need for existence and computation of central dispersion measures, wherever the underlying ball property is satisfied.

\begin{example}
Consider a circle $\mb{S}^1_r$ with radius $r$. The arc metric $d_A$ and Euclidean metric $d_E$ are related in the following way
  \begin{align*}
      d_E(x,y) = 2r \sin\big(\frac{d_A(x,y)}{2r}\big),
  \end{align*}
  where $x,y \in \mb{S}^1_r$. Note that this is a contraction since $d_E(x,y) \le d_A(x,y)$. 
  The derivative  $\frac{d f(s)}{d s} = \cos(\frac{s}{2r})$ of $f(s) = 2r \sin(\frac{s}{2r})$ can be seen as the rate of change of the Euclidean metric with respect to the actual movement along the circle. If $s \to 0$ then it is close to 1. If $s$ gets larger, the derivative drops to zero where the points are antipodal, i.e., where $s=r\pi$. 
  Gromov \cite{Gromov} showed that the distortion of embedding a closed loop into Euclidean space is at least $\pi/2$. Hence, the circle attains Gromov's lower bound and represents the least distorted configuration possible for a closed loop embedded in a Euclidean space.  However, if probability mass is concentrated around the antipodal points, a statistic based on the Euclidean embedding will be flawed in distinguishing the difference between antipodal versus near antipodal points. 

\end{example}

If \autoref{thm:isombedSch} fails, and a Hilbert space embedding is no longer compatible with the geometric structure,  it is of interest to construct a more general separating class of functions. In particular, we note the following: \begin{thm}\label{thm:isomdeterminedbyballs}
      The collection of complete, separable metric spaces that are isometric to a subset of a Hilbert space is a strict subclass of the collection of metric spaces for which the Borel  probability  measures are determined by their values on balls. 
\end{thm}

\begin{proof}

If we have an isometric embedding $\phi :\mc{X} \to H$, then for any $r >0$ and any $y \in \mc{X}$, we can define the kernel function \[k_{y,r}(x) = e^{-r d_H(\phi(x), \phi(y))} = e^{-r d_{\mc{X}}(x,y)}\] for any $r >0$. By 
\autoref{thm:isombedSch} and the preceding discussion, we know that this kernel is strictly positive definite and thus $\Gamma_{\varphi}$ forms a separating class. The statement now follows immediately from the equivalence established in  \autoref{lem:cdfslapchar}.

\end{proof}

\section{Foundations for measure-theoretic inference} \label{sec:sec3}

 In practice, the researcher lacks full information on the measures $\mu$ and $\nu$, but can construct empirical versions from realizations of a process that is (at least locally) characterized via an ergodic measure-preserving shift transformation. Depending on the properties of the space, a class $\Gamma$ that separates measures may be so rich that supremum-based estimation over it becomes unfeasible.  In such cases, $L^p$-based discrepancy measures may be preferred. If a natural reference measure is not a priori available, one may require specification and estimation of an asymptotically mean stationary measure \cite{graykieffer1980}. In this section, we introduce multi-parameter weak invariance principles in the context of nonstationary Polish-valued stochastic processes. These can be used, among other things, to estimate the integral probability metric as well as $L^p$-based metrics for the various separating classes $\Gamma$ that were introduced above.

We assume that we observe a sample $\{\mb{X}\}_{t=1}^T$ from a stochastic process $\{\mb{X}_{t}: t \in \znum\}$ defined on a common probability space $(\Omega, \F, \pr)$ where each random variable $\mb{X}_t$ takes values in a Polish space $(S,d_S)$. As we specifically want to allow for nonstationary dynamics in the probabilistic structure, we need to consider a local asymptotic framework that accommodates such behavior, so-called \textit{infill asymptotics}. For this, we use the framework of locally stationary Polish-valued time series as introduced in \cite{vDB25}, which builds on \cite{dahlhaus1997,SR06}. We consider a triangular array of processes $(\mb{X}_{t,T}: t= 1,\ldots, T: T \in \mathbb{N})$ indexed by the sample size $T \in \mb{N}$. The double-indexed process is extended on $\znum$ by setting $\mb{X}_{t,T}=\mb{X}_{1,T}$ if $t \le 0$ and $\mb{X}_{t,T}=\mb{X}_{T,T}$ if $t\ge T$. 
\begin{definition}\label{def:Localstationarity}
Let $(\mb{X}_{t,T}: t \in \znum, T \in \nnum)$ be an $S$-valued stochastic process. $(\mb{X}_{t,T}: t \in \znum, T \in \nnum)$ is \textit{locally stationary} if there exists an $S$-valued process $(\mb{X}_t(u): t \in \znum, u \in [0,1])$  that is stationary for each fixed $u=t/T \in [0,1]$ such that
\[
\E\Big[ d_S\big(\mb{X}_{t,T}, \mb{X}_t(\frac{t}{T}) \big) \wedge 1 \Big]=O(T^{-1}) \quad \text{and} \quad \E\Big[d_S\big(\mb{X}_{t}(u), \mb{X}_t(v)\big) \wedge 1 \Big]=O(|u-v|), \tageq \label{eq:locstatp}
\] 
uniformly in $t=1, \ldots, T$ and $u,v \in [0,1]$.
\end{definition}

For the unfamiliar reader, it is worth emphasizing that if the underlying process is stationary (such as an iid sequence) then $\mb{X}_{t,T}=\mb{X}_t$ for all $t=1,\ldots, T$ and $T \in \mathbb{N}$. In this case, infill asymptotics simply coincides with classical asymptotics. Hence,  processes that fit within the infill asymptotic framework encompass the class of stationary processes.

Let $\mc{M}_{+, \le 1}(S)$ denote the space of nonnegative Borel measures on $S$ with total mass bounded by $1$, equipped with the topology of weak convergence. Recall that this topology on $\mc{M}_{+, \le 1}(S)$ is  metrized by the Bounded Lipschitz metric $d_{BL}$ induced by the standard unit ball $\mathscr{F}_{BL}$, which is given by
\[
d_{BL}(\mu,\nu) = \sup\Big\{\bv \int_S f d\mu -\int_S f d\nu \bv: \|f\|_{BL} \le 1\Big\}
\]
where $\|f\|_{BL} = \|f\|_{\infty} +\text{Lip}(f)\le 1$ \cite{dudley2002}. 
Given a sample  $(\mb{X}_{t,T})_{t=1}^T$, we can compute sequential random measures $(\mu_{T,u}: u \in [0,1])_{T \ge 1}$ by
\[
\mu_{T,u}(\omega,A)= \frac{1}{T} \sum_{t=1}^{\flo{uT}}\delta_{\mb{X}_{t,T}(\omega)}(A) \quad u \in[0,1], A \in \mc{B}(S), \tageq \label{eq:empmu}
\] 
Note that for any $\omega \in \Omega$, this sequence lies in  $D([0,1], \mc{M}_{+, \le 1}(S))$, the Skorokhod space of c{\`a}dl{\`a}g functions endowed with the Skorokhod topology $J_1$. 
Under the assumptions used in this paper, the limiting process has almost surely continuous sample paths: the corresponding sequence of measures $(\tilde{\mu}_u: u \in [0,1])$ is given by
\[
\tilde{\mu}_u(A)=\int_0^u \pr(\mb{X}_0(\eta)\in A) d\eta, \quad A \in \mc{B}(S). \tageq \label{eq:popmu}
\]

\begin{thm}
    \label{thm:ergthm}
Under the assumptions of \autoref{def:Localstationarity}, the sequence $(\mu_{T,u}: u \in [0,1])_{T \ge 1}$ converges weakly in $D([0,1],\mc{M}_{+, \le 1}(S))$ to the collection  $(\tilde{\mu}_{u}: u \in [0,1])$ which belongs to $C([0,1], \mc{M}_{+, \le 1}(S))$ almost surely.
\end{thm}
Here $C([0,1],\mc{M}_{+, \le 1}(S))$ is endowed with the topology induced by the metric $\rho(\mu,\nu)=\sup_{u\in[0,1]} d_{BL}(\mu_u, \nu_u)$. In particular, since the limit is continuous almost surely, an application of the continuous mapping theorem yields  the following convergence result:
\begin{Corollary}
    Under the conditions of \autoref{thm:ergthm}, as $T \to \infty$, we have
\[
\sup_{u \in [0,1]} d_{BL}(\mu_{T,u},\tilde{\mu}_u) \overset{p}{\to} 0. 
\]

\end{Corollary}

\begin{proof}[Proof of \autoref{thm:ergthm}]
We first show the convergence of the fidis. If $u <1/T$, the measure is zero, so let $u \ge 1/T$.  
Since $S$ is Polish, $\mu_{T,u} \Rightarrow \tilde{\mu}_u$ is equivalent to 
\[
\int_S f(x) \mu_{T,u}(\omega, dx) \to \int_S f(x) \tilde{\mu}_{u}(dx)
\]
for all $f \in BL(S)$ equipped with $\|f\|_{BL} = \|f\|_{\infty} +\text{Lip}(f)\le 1$. Furthermore, separability of $(S,d_S)$  implies we can find a sequence of functions $\{f_n\} \subset BL(S)$ that is dense in the unit ball under the uniform norm, i.e., 
for every $g \in BL(S)$ and every $\epsilon>0$, $\exists n \in \mb{N}$ such that $\|g-f_n\|_{\infty} <\epsilon$. Below, let $k_T: [0,1] \to \mb{N} \cup \{0\}$ be the sequence given by $k_T(u) = \flo{uT}$.
Let  $f \in BL(S)$ with $\|f\|_{BL} \le 1$, then by the triangle inequality
\begin{align*}
   & \E\bv \frac{1}{k_T(u)} \sum_{t=1}^{k_T(u)} f(\mb{X}_{t,T})- \frac{1}{k_T(u)} \sum_{t=1}^{k_T(u)} f(\mb{X}_{t}(t/T))
\bv  
\\& \le \E\bv  \frac{1}{k_T(u)} \sum_{t=1}^{k_T(u)} f(\mb{X}_{t,T})- \frac{1}{k_T(u)} \sum_{t=1}^{k_T(u)} f(\mb{X}_{t}(t/T))\bv  1_{| f(\mb{X}_{t,T})-f(\mb{X}_{t}(t/T))|\le \epsilon} \\&+ \E\bv \frac{1}{k_T(u)} \sum_{t=1}^{k_T(u)} f(\mb{X}_{t,T})- \frac{1}{k_T(u)} \sum_{t=1}^{k_T(u)} f(\mb{X}_{t}(t/T)) \bv 1_{|f(\mb{X}_{t,T})-f(\mb{X}_{t}(t/T))|> \epsilon}
\\& \le  \epsilon + 2\max_{1\le t\le T}\pr( d_S(\mb{X}_{t,T},\mb{X}_{t}(t/T))> \epsilon) \tageq\label{eq:decomperg}
\end{align*}
since $|f(x)-f(y)| \le \|f\|_{BL}d_S(x,y)$.  Hence, taking $\epsilon_T = T^{-\alpha}$ where $\alpha \in (0,1)$ will ensure that this goes to 0 as $T \to \infty$. For the next step we use a standard technique. Let $D_{L,k,l} = \Big\{t \in \znum : \frac{t}{k} \in \big(\frac{l-1}{2^L},\frac{l}{2^L}\big] \Big\}$. It follows similar to a  decomposition as in \eqref{eq:decomperg} that 
\begin{align*}
    &\E\bv \frac{1}{k_T(u)} \sum_{t=1}^{k_T(u)} f(\mb{X}_{t}(t/T)) - \frac{1}{2^L} \sum_{l =1}^{2^L} \frac{1}{|D_{L,k_T(u),l}|} \sum_{t \in D_{L,k_T(u),l}}f(\mb{X}_{t}(\frac{l}{2^L}u)) \bv
    \\& = 
O(\frac{2^{L}}{k_T(u)}) +  O(2^{-\alpha L}(1+o(1))).
\end{align*}
Denote the random  measure 
\[
\tilde{\mu}_{T,L,u}(\omega,A)=\frac{1}{2^L}\sum_{l =1}^{2^L} \frac{u}{|D_{L,k_T(u),l}|} \sum_{t \in D_{L,k_T(u),l}}\delta_{\mb{X}_{t}(\omega,\frac{l}{2^L}u)}(A).
\]
Then $\sup_{u \ge 1/T}|\frac{u}{k_T(u)}-\frac{1}{T}]=O(1/T)$, 
together with the preceding argument yields that for all $f \in BL(S)$ with $\|f\|_{BL} \le 1$
\[
\lim_{L \to \infty}\limsup_{T \to \infty}\pr(d_{BL}(\mu_{T,u}, \tilde{\mu}_{T,L,u}) \ge \epsilon) =0
\]
for all $\epsilon>0$. Now, by the ergodic theorem, for any $f \in BL(S)$ with $\|f\|_{BL} \le 1$  
\[
\frac{1}{|D_{L,{k_T},l}|} \sum_{t \in D_{L,k_T,l}}f\big(\mb{X}_{t}(\omega,\frac{l}{2^L}u)\big) \to \E f\big(\mb{X}_0(\frac{l}{2^L}u)\big)
\]
as $T \to \infty$ a.s. If we denote by $\Omega_n \subset \Omega$ the set on which this convergence holds for the element $f_n$ of the sequence $\{f_n\}$, then since $\pr(\cap_n \Omega_n)=1$,  $\tilde{\mu}_{L,k_T(u)} \Rightarrow \tilde{\mu}_{L,u}$ almost surely as $T \to \infty$.  
By a Riemann approximation and a  change of variables
\[
\lim_{L \to \infty} \frac{u}{2^L} \sum_{l=1}^{2^L} \pr({\mb{X}_{0}(\frac{l}{2^L}u)} \in A) =\int_0^u  \pr(\mb{X}_0(\eta) \in A) d\eta. 
\]
Hence, using iterated limits (see e.g., \cite{bil68}) $d_{BL}(\mu_{T,u}, \tilde{\mu}_u) \overset{p}{\to} 0 $ as $T \to \infty$. This establishes that $\mu_{T,u_i} \Rightarrow \tilde{\mu}_{u_i}$ in probability for each $u_i \in [0,1]$, $i =1, \ldots, J$, where $J$ is finite. Since weak convergence on the product space $\prod_{i=1}^J\mc{M}_{+, \le 1}(S)$ is metrized by 
\[
d^J_{BL}({\bf \mu},{\bf \nu}) = \max_{1\le i \le J} d_{BL}( \mu_{i},\nu_i),
\]
the pointwise weak convergence in probability implies that $\max_{1\le i \le J} d_{BL}( \mu_{T,u_i},\tilde{\mu}_{u_i}) \overset{p}{\to}0$. Thus
\[
(\mu_{T,u_1}, \ldots, \mu_{T,u_J}) \Rightarrow (\tilde{\mu}_{u_1}, \ldots, \tilde{\mu}_{u_J})
\]
in probability as $T \to \infty$ for any finite set of coordinates $u_1, \ldots u_J \in [0,1]$. It now remains to show that for all
$\epsilon> 0$, 
\[
\lim_{\delta\downarrow 0} \limsup_{T \to \infty} \pr(\omega(\mu_T, \delta) > \epsilon) =0
\]
where 
$\omega(\mu_T, \delta) =\sup_{|t-s| \le \delta} d_{BL}(\mu_{T,t},\mu_{T,s})$.
For $|u_2-u_1| \ge 1/T$  and for any $\alpha>1$, we find 
\begin{align*}
    \E \sup_{\|f\|_{BL}\le 1} \frac{1}{T}\bv \sum_{t=\flo{u_1 T}+1}^{\flo{u_2 T}}f(\mb{X}_{t,T}) \bv^{\alpha} \le (u_2-u_1+1/T)^\alpha \le (2(u_2-u_1))^{\alpha}.
\end{align*}
 If $|u_1-u_2| \le 1/T$, then $d_{BL}(\mu_{T,u_2},\mu_{T,u_1}) \le 1/T$. Therefore, 
\[\pr(\sup_{|t-s|< 1/T} d_{BL}(\mu_{T,t},\mu_{T,s}) >\epsilon) \le \pr(1/T >\epsilon) \to 0\]
as $T \to \infty$. Hence the result follows from Theorem 1 of \cite{DZ08}.
\end{proof}

We assume the following representations for the family of processes.

\begin{assumption}\label{as:rep}
    Assume that $\{\mb{X}_{t,T}\}$ is a locally stationary process  as defined in \autoref{def:Localstationarity}  and that the random functions admit representations
\[
\mb{X}_{t,T}:=f\big(t,T,\mathfrak{g}_t \big),  \quad \mb{X}_{t}(u):=\ddot{f}\big(u,\mathfrak{g}_t \big),\tageq\label{eq:bershift}
\]
where $\mathfrak{g}_t =(\varepsilon_t, \varepsilon_{t-1}, \ldots)$ for an iid sequence $\{\varepsilon_t:  t\in\mathbb{Z}\}$  of random elements taking values in some measurable space $\mathfrak{G}$, and where  $f\colon\znum\times \mathbb{N} \times \mathfrak{G}^{\infty} \to S$, and $\ddot{f}\colon [0,1] \times \mathfrak{G}^{\infty} \to S$ are measurable functions.
\end{assumption}

Additionally, we need a measure of dependence. We consider innovation-level coupling as in \cite{vDB25} \citep[see also][]{Wiener1958,rosenblatt1959,berbee1979,Wu2005}. That is, let $\G_t = \sigma(\varepsilon_s, s\le t)$  and denote $\G_{t,m}=\sigma(\varepsilon_t, \ldots, \varepsilon_{t-m+1}, \varepsilon_{t-m}^\prime, \varepsilon^\prime_{t-m-1}, \ldots)$ where $\{\varepsilon^\prime_t:  t\in\mathbb{Z}\}$ is an independent copy of $\{\varepsilon_t:  t\in\mathbb{Z}\}$. In addition, we let $\G_{t,\{t-j\}} := \sigma(\varepsilon_t, \ldots, \varepsilon_{t-j+1}, \varepsilon^\prime_{t-j}, \varepsilon_{t-j-1}, \ldots)$ but where element $\varepsilon_{t-j}$ is replaced by an independent copy. Then 
 $\mb{X}_{m,t}$ is an $m$-dependent coupled version of $\mb{X}_{t}$. 
 Similarly, $\mb{X}_{\{t-j\},t}$
   is a coupled version independent of innovation
$\varepsilon_{t-j}$ of $\mb{X}_{t}$.  The strength of dependence may then be measured via  $$
\delta^p_p(m):=\sup_t\E d^p_S\big(\mb{X}_t, \mb{X}_{m,t}\big)  \text{ or }  \varrho^p_{p}(j):=\sup_t\E d^p_S\big(\mb{X}_t, \mb{X}_{\{t-j\},t}\big)$$ for some $p \ge 1$. We shall also make use of weaker versions where appropriate:
\begin{align*}
    \delta_0(m):=\sup_t\E \Big[d_S(\mb{X}_t, \mb{X}_{m,t}) \wedge 1\Big]  \text{ or }  \varrho_{0}(j) :=\sup_t \E\Big[d_S\big(\mb{X}_t, \mb{X}_{\{t-j\},t}\big)\wedge 1\Big].\tageq \label{eq:delta0}
\end{align*}
Note that the latter type of assumption implies that $d_S(\mb{X}_t, \mb{X}_{m,t})=O_p\big(\delta_0(m)\big)$ and $d_S\big(\mb{X}_t, \mb{X}_{\{t-j\},t}\big)=O_p\big(\varrho_0(j)\big)$, respectively.

\subsection{Weak invariance principles under general metric geometries }\label{sec:general}

Here, weak invariance principles are introduced that can be used as a building block for measure-theoretic inference when \autoref{thm:isombedSch} may fail, but when the ball property still holds. In the latter case, it is necessary to explicitly account for the geometric complexity. 
Not surprisingly, stronger assumptions on the tail decay are required in the case of exponential growth than in the case of polynomial growth.
We focus on the separating class $\Gamma_{\varphi}$ as defined in \autoref{lem:cdfslapchar}, but the discussion below immediately applies to $\Gamma_{\phi}$, with a bit more notational clutter. Suppose $\mu_T$ and $\nu_T$ are (consistent) empirical versions of $\mu$ and $\nu$. Let 
\begin{align*}
      V(y,r) = \int f_{y,r} (x) \mu(dx) - \int f_{y,r} (x) \nu(dx), \quad y \in S, r>0
    \end{align*}
then 
\begin{align*}
      \hat{V}_T(y,r) = \int f_{y,r} (x) \mu_T(dx) - \int f_{y,r} (x) \nu_T(dx), \quad y \in S, r>0
    \end{align*}
can be used as a basis for quantifying discrepancies between   measures $\mu$ and $\nu$. 
For example, 
 \begin{align*}
   \mc{D}(\mu_T,\nu_T)=  \sup_{y \in S, r \in [0,b]}\bv  \hat{V}_T(y,r) \bv \tageq \label{eq:Was1}
    \end{align*}
    can be seen to serve as an empirical version of \eqref{eq:ipsm}. 
  The restriction to $[0, b] \subset [0,\infty)$ is without loss of generality by the identity theorem, since the function $r\mapsto f_{y,r}(x)$ is analytic for all $r >0$ 
  \citep[see e.g.,][]{kallenberg1997,feller1971}. 
We may restrict to a countable dense subset of $[0, b]$. Furthermore, since $S$ is second countable, we can also take a countable dense subset $D \subset S$ to evaluate $y$, provided we have sufficient information about $S$. Supremum-based statistics generally are powerful in detecting sharp, localized, high-frequency
abnormalities in the data. If the data are dominated by high-frequency noise or are subject to global, low-frequency changes, an $L^p$-based statistic is often more robust. 

For this, however, we must specify a probability measure $\beta\in \mc{P}(S)$. An $L^p$-based statistic is then given by
\begin{align*}
      \mc{L}(\mu_T,\nu_T) = \int_S \int_0^b\bv  \hat{V}_T(y,r) \bv^p dr\,\beta(dy).  \tageq \label{eq:Lp}
    \end{align*}
For example, one could consider the uniform measure on $(S,\mc{S})$ if well-defined, or the measure $\beta=\frac{1}{2}(\mu+\nu)$, which has to be estimated from the sample. For the latter, we will invoke \autoref{thm:ergthm} as established above.

For the purpose of detecting localized behavior in a single process, interest also lies in sequential versions of $\hat{V}_T(y,r)$, given by
\[
\hat{V}_T(u_1,u_2,y,r) = u_2 \int_S f_{y,r} (x) \mu_{T,u_1}(dx) - u_1\int_S f_{y,r} (x) \mu_{T,u_2}(dx),\quad u_1, u_2 \in [0,1]\]
where we use the notation introduced in \eqref{eq:empmu}-\eqref{eq:popmu}.

\begin{example}\label{ex:1}
Consider testing for time-varying behavior in the marginals
\[
H_0: \mu_t = \mu \text{ for all $t \in \znum$}\text{\quad vs \quad} H_1: \mu_t \neq \mu \text{ for some $t \in \znum$}
\] 
If $\Gamma_{\varphi}$ is separating, then 
$V(u,1,y,r)=0$ for all $y \in S, r>0$ implies that the marginals are the same. A test can be constructed based on an appropriately scaled version of 
 \begin{align*}
  \sup_{u \in [0,1],y \in S, r \in [0,b]}\bv  \hat{V}_T(u,1,y,r)\bv. \tageq \label{eq:Was}    \end{align*}

\end{example}

In order to derive the distributional properties given a sample $(\mb{X}_{t,T})_{t=1}^T$, we analyze the empirical  process $\{\zeta_T(u,y,r): u \in [0,1], y \in S, r \in [0, b]\}_{T \ge 1}$ defined as follows:
    \[
    {\zeta}_T(u,y,r) = \frac{1}{\sqrt{T}} \sum_{t=1}^{\flo{uT}} f_{y,r}(\mb{X}_{t,T}).
    \]
    The demeaned kernel will be denoted by   $\tilde{f}_{y,r}(\mb{X}_{t,T})= f_{y,r}(\mb{X}_{t,T})-\E [f_{y,r}(\mb{X}_{t,T})]$.
As a process indexed only by localized time, we obtain the following result. 
\begin{thm}\label{thm:fidismetric}
   Assume that $\{\mb{X}_{t,T}\}$ is a Polish-valued locally stationary process with auxiliary process $\{\mb{X}_t(u): t\in \znum, u \in [0,1]\}$ as defined in \autoref{def:Localstationarity} with respective representations as in \autoref{as:rep}, and such that $\delta_0(m) \to 0$ as $m \to \infty$ for both processes. In addition, $\sup_T \sum_{j=0}^{\infty}\sup_{t \in [T]} \theta^{\mb{X}}_{t}(j) <\infty$ and $\sup_u \sum_{j=0}^{\infty}\sup_{t \in \znum} \theta^{\mb{X}(u)}_{t}(j) <\infty$,
where
\begin{align*}
\theta^W_{t}(j)=
 \sup_{m\ge 1} \|P_{t-j}({f}_{y,r}(W_{m,t}))\|_{\rnum,2} 
 \tageq \label{eq:thetheta}
 \end{align*} 
with $f_{y,r} \in \Gamma_\varphi$. Then for any finite sets $J_S \subset S$ and $J_R \subset [0,b]$, the process $$\{\zeta_T(u,y_i,r_j)-\E\zeta_T(u,y_i,r_j) : u \in [0,1], y_i \in J_S, r_j \in J_R\}$$ converges in Skorokhod topology to a zero mean Gaussian process $\{\zeta(u,y,r): u \in [0,1],  y \in J_S$, $r \in J_R\}$ with covariance structure
\[
\mathrm{cov}\big({\zeta}(u_1,y_1,r_1), {\zeta}(u_2,y_2,r_2)\big) = \int_0^{u_1 \wedge u_2} \sigma_\eta( y_1, r_1, y_2, r_2) d\eta
\]
where $\sigma_\eta( y_1, r_1, y_2, r_2)$ is the local long-run covariance function at $\eta \in [0,1]$, given by  
\[
\sigma_\eta( y_1, r_1, y_2, r_2)= \sum_{h \in \znum} \mathrm{cov} \Big(\tilde{f}_{y_1, r_1}\big(\mb{X}_0(\eta)\big),\tilde{f}_{y_2, r_2}\big(\mb{X}_h(\eta)\big) \Big)<\infty.
\]
\end{thm}
Here $P_{j}(X):= \E[X|\G_{j}]-\E[X|\G_{j-1}]$ for an integrable random variable $X$. 
\autoref{thm:fidismetric} is sufficient to ensure that the finite-dimensional distributions of $\zeta_T$ converge to those of $\zeta$. In order to show that the limiting process has almost surely continuous sample paths, the corresponding family of probability measures must be relatively compact in the weak topology of probability measures on $C(E,\rnum)$ equipped with the supremum norm. For this, we shall focus on the setting where $S$ is compact such that $E = [0,1]\times S \times [0,b]$ endowed with a (pseudo)-metric $d_E$ is itself compact. We will consider the metric $d_E((u,y,r),(u^\prime,y^\prime, r^\prime))= \max\{|u-u^\prime|^{1/2},d_S(y,y^\prime)^{1/2},|r-r^\prime|^{1/2}\}$. Then, it suffices to show that for some $x \in E$, $\{\zeta _T(x); T \ge 1\}$ is relatively compact as a sequence of real-valued random variables and that for each $\epsilon>0$ and $\eta>0$, there exists a $\delta>0$ and $T_0 \in \mb{N}$ such that 
\begin{align*}
    \pr\Big(\sup_{d_E(x,x^\prime) \le \delta} | \zeta_T(x)-\zeta_T(x^\prime)| \ge \epsilon\Big) \le \eta, \quad T \ge T_0. 
\end{align*}
The first condition follows from the above theorem. For the second condition, we recall that a Young function $\psi: \rnum_+ \to \rnum_+$ is convex, increasing and satisfies $\psi(0)=0$, $\lim_{x \to \infty} \psi(x) = \infty$, and that 
 $N(E,d_E, \epsilon)$ denotes the smallest number of open balls of radius  $\epsilon>0$ in $d_E$ which form a covering of $E$. For the case of exponential growth, we define for any random variable $X$ and $\sigma$-algebra $\tilde{\F} \subseteq\F$, the following random variable
\[
\|X\mid\tilde{\F}\|_{\psi} = \inf\Big\{k \in L^0(\tilde{\F}), k >0 \text{ a.s. } : \E\Big[\psi\Big(\frac{|X|}{k}\Big)\bv \tilde{\F}\Big]\le 1\Big\}
\]
Note that this reduces to the Orlicz norm if we take $\F=\{\emptyset, \Omega\}$. We furthermore recall that the essential supremum is given by $\text{ess\,sup}(X)= \inf\{M: \pr(X > M)=0\}$.

\begin{thm}[Polynomial growth]\label{thm:conditionsfortightpol}
   Suppose that $N(E, d_E, \epsilon) \sim \epsilon^{-\alpha}$. If for some $p >\alpha$, and all $T \ge 1, y \in S$, $r \in [0,b]$,
    \[
\sum_{j=0}^\infty j \sup_{t \in [T]} \| P_{t-j}(f_{y,r}(\mb{X}_{t,T}))\|_{\rnum,p} <\infty, \tageq \label{eq:sumtightcompol}
    \]
    then 
    \begin{align*}
        \sup_{T \ge 1} \|\zeta_T(u,y,r)-\zeta_T(u^\prime,y^\prime, r^\prime)\|_{\psi} \le K\big( |u-u^\prime|^{1/2} \vee d_S(y,y^\prime)^{1/2}\vee |r-r^\prime|^{1/2}\big), \tageq \label{eq:boundtightcompol}
    \end{align*}
    where $\|\cdot\|_{\psi}$ denotes the Orlicz norm  with $\psi(x)= |x|^p$. 
\end{thm}

\begin{thm}[Exponential growth]\label{thm:conditionsfortight}
   Suppose that $N(E, d_E, \epsilon) \sim e^{\epsilon^{-\alpha}}$ for some $\alpha<2$. If for all $T \ge 1, y \in S$, $r \in [0,b]$,
    \[
\sum_{j=0}^\infty j \sup_{t \in [T]} \text{ess\,sup}\| P_{t-j}(f_{y,r}(\mb{X}_{t,T}))\mid\G_{t-j-1}\|_{\psi} <\infty, \tageq \label{eq:sumtightcom}
    \]
     then 
    \begin{align*}
        \sup_{T \ge 1} \|\zeta_T(u,y,r)-\zeta_T(u^\prime,y^\prime, r^\prime)\|_{\psi} \le  K\big( |u-u^\prime|^{1/2} \vee d_S(y,y^\prime)^{1/2}\vee |r-r^\prime|^{1/2}\big),\tageq \label{eq:boundtightcom}
    \end{align*}
    where $\|\cdot\|_{\psi}$ denotes the Orlicz norm with $\psi(x)= e^{x^2}-1$. 
\end{thm}

In the above theorems, the increments are controlled in the appropriate Orlicz norm  $\|\cdot\|_\psi$. 
This follows from the fact that irregularity of the sample paths of a stochastic process $(U(x): x \in E)$ indexed by a metric space $E$ is controlled in the following sense \cite{ledouxtalagrand}: 
\[
\bignorm{\sup_{d_E(x,x^\prime)\le \delta}|U(x)-U(x^\prime)|}_{\psi} \le C \int_0^\delta \psi^{-1}(N(E,d_E,\epsilon)) d\epsilon,
\]
which shows how the growth rate of the covering number of the product space determines the allowed maximum irregularity of the sample paths. This is easily verified to be determined by the coordinate space in the product space $E$ with the fastest growth rate. Note that for a closed interval $[a,b]$ endowed with the Euclidean metric, the covering number only grows at a polynomial rate for which bounding the increments  with $\psi(x) =x^p$ would be adequate. In essence, if the geometric structure of the metric space $(S,d_S)$ is more complex, we are bound to impose stronger conditions on the control of the increments as provided in \autoref{thm:conditionsfortight}.

\begin{Corollary}
    \label{cor:fcltmetric}
If the conditions of \autoref{thm:fidismetric} together with \autoref{thm:conditionsfortightpol} or \autoref{thm:conditionsfortight} hold, then 
    \begin{align*}
      &  \Big\{\zeta_T(u,y,r)-\E\zeta_T(u,y,r): u \in [0,1], y \in S, r \in [0,b]\Big\} \\&\rightsquigarrow  \Big\{\zeta(u,y,r): u \in [0,1], y \in S, r \in [0,b]\Big\} 
    \end{align*}
    where $\Big\{\zeta(u,y,r): u \in [0,1], y \in S, r \in [0,b]\Big\} $ is a zero mean Gaussian process with almost surely continuous sample paths and covariance structure as defined in \autoref{thm:fidismetric}.  
\end{Corollary}

\begin{Remark}
    If $(E,d_E)$ is locally compact and Polish, the weak invariance principles can be seen to extend to $D(E, \rnum)$, provided the latter is endowed with the global Skorokhod topology. 
\end{Remark}

\begin{Proposition}\label{prop:infsumrepball}
    We can represent the process as  
\[
\zeta(u,y,r):=\sum_{\ell=1}^{\infty} \int_0^{u} \tilde{g}_\ell(\eta,y,r) d\mathbb{B}_\ell(\eta),  \tageq\label{eq:Gauslim}\]
where $\{\mb{B}_{\ell}\}_{\ell \in \mb{N}}$ is a sequence of independent standard Brownian motions and where the functions $\tilde{g}_{\ell}$ are such that 
\[
\sigma_\eta(y_1,r_1,y_2,r_2)= \sum_{\ell=1}^\infty \tilde{g}_{\ell}(\eta, y_1,r_1)\tilde{g}_{\ell}(\eta, y_2,r_2). 
\]
Under stationarity this reduces to $\zeta(u,y,r):=\sum_{\ell=1}^{\infty} \tilde{g}_\ell(y,r)   \mathbb{B}_\ell(u)$.
\end{Proposition}
The above proposition follows from Mercer's decomposition since the continuity of $\sigma_\eta(\cdot)$ on $[0,1] \times E \times E$ follows directly from the analyticity of the separating class functions combined with the uniform convergence guaranteed by the bounded projection series in \autoref{thm:fidismetric}. This weak invariance principle can be used as a basis for various statistics. By the continuous mapping theorem, we find for \autoref{ex:1}:
\begin{Corollary}\label{cor:fcltmetricsup}
   If the conditions of \autoref{cor:fcltmetric} hold, then 
  \begin{align*}
&\sup_{u \in [0,1], y \in S, r \in [0,b]}\sqrt{T}  |\hat{V}_T(u,1,y,r) -\E \hat{V}_T(u,1,y,r)| \\& =\sup_{u \in [0,1], y \in S, r \in [0,b]} | \zeta_T(u,y,r)-\E \zeta_T(u,y,r)- u(\zeta_T(1,y,r)-\E \zeta_T(1,y,r))|
   \\& \Rightarrow \sup_{u \in [0,1], y \in S, r \in [0,b]} |\zeta(u,y,r)-u \zeta(1,y,r)|. 
  \end{align*}
\end{Corollary}

To provide a basis for constructing $L^p$-based tests as in \eqref{eq:Lp}, we require a suitable choice for the measure $\beta$ on $(S,\mc{S})$. A natural choice is given below, where we also discuss the properties of the supremum-based statistics if the sample itself provides the best information about the geometric structure of the space $S$.

\subsubsection{The intrinsic resolution of the sample}\label{subsub:estimationS} 

Whether we can replace the supremum over $S$ in \autoref{cor:fcltmetricsup} with a plug-in estimator based on the sample without changing the limiting distribution depends again on the geometric complexity. Let $\hat{S}=\{\mb{X}_{1,T}, \ldots, \mb{X}_{T,T} \}$ (or a suitable subcollection that still grows to infinity with $T$).

\begin{thm}\label{thm:consisgrid}
    Assume that $\inf_{m,t,T,x}\,\pr\big(\mb{X}_{m,t,T} \in B_{d^{1/2}_S}(x,\epsilon)\big) \ge \psi(\epsilon)$, where $\mb{X}_{m,t,T}$ is the $m$-dependent coupled version of $\mb{X}_{t,T}$. If $\psi(\epsilon) \sim \epsilon^\alpha$, where $0<\alpha<1$, and the dependence coefficients satisfy $\delta_p(j)=O(1/j^\gamma)$, where $\gamma >\frac{2}{1-\alpha}$ and $p >\frac{2\alpha}{\gamma(1-\alpha)}$, then under the conditions of \autoref{thm:conditionsfortightpol}:
     \begin{align*}
 &\sup_{u \in [0,1], y \in \hat{S}, r \in [0,b]}\sqrt{T}  |\hat{V}_T(u,1,y,r) -\E \hat{V}_T(u,1,y,r)|  
   \\& \Rightarrow \sup_{u \in [0,1], y \in S, r \in [0,b]} |\zeta(u,y,r)-u \zeta(1,y,r)|. 
  \end{align*}
\end{thm}
The small ball exponent $\alpha$ can be seen as the intrinsic dimension of the data; \autoref{thm:consisgrid} explicitly exposes the geometric threshold required to avoid spatial sparsity. Because the local ball radius contracts at a rate of $O(T^{-1/\alpha})$, the availability of $T$ sample points requires $\alpha < 1$ for the sample grid to remain dense in $S$.
Indeed, this forces the sample size to dominate the local contraction, ensuring that neighborhoods collapse slowly enough to bypass the severe spatial sparsity typical of higher dimensions. Note that if the process is independent, the assumption on the dependence coefficients can be dropped. 

This "curse of dimensionality" is something that can be avoided in the following scenarios:
\begin{enumerate}[label=(\roman*)]
\item We have a sample that is dense in spatial direction, i.e., if for each $\mb{X}_t$ a point cloud or net is available \citep[see e.g.,][]{vDB25}. We note the weak invariance principles are straightforward to adapt to this scenario using techniques of the aforementioned work; 
\item The space is isometric to a subset of a Hilbert space  as discussed in the next section; or 
\item By considering alternative metrics: an $L^p$-based statistic instead of a supremum-based statistic, which integrates out the worst-local sparsity gap that plagues a supremum-based statistic.
\end{enumerate}

For the last scenario, we need a stable estimator for measure $\beta$.

\begin{thm}\label{thm:consisgridlp}
    Assume that $n = n_T \to \infty$ as $T \to \infty$, and that $\bar{\mu}_T= \frac{1}{T} \sum_{t=1}^T \delta_{\mb{X}_{t,T}}$ converges weakly to an element $\bar{\mu}$ of $\mathcal{P}(S)$. Then, under the conditions of \autoref{thm:conditionsfortightpol} or of \autoref{thm:conditionsfortight} for $p > 0$, we have
    \begin{align*}
    & \frac{1}{b}\int_0^b \int_S   \int_0^1 T^{p/2}\left| \int_S \tilde{f}_{y ,r} (x) \mu_{T,u}(dx) - u\int_S \tilde{f}_{y,r} (x) \mu_{T,1}(dx)\right|^p  \Delta_{1,n}(du) \bar{\mu}_T(dy) \Delta_{b,n}(dr)\\
    & \quad \Rightarrow \frac{1}{b}\int_0^b\int_{S} \int_0^1  \left|\zeta(u,y,r)-u \zeta(1,y,r)\right|^p  du \, \bar{\mu}(dy) dr,
    \end{align*}
    where $\Delta_{q,n} = \frac{1}{n}\sum_{i=1}^{n} \delta_{z_i}$ for $z_i \in G_{q,n}:=\{z_1, \ldots, z_n\}$, and where $G_{q,n}$ is a set of evenly spread grid points in the interval $[0,q]$.
\end{thm}

Note that the key requirement for such a measure $\bar{\mu}$ to exist is that the empirical measure converges to a deterministic probability measure. Sufficient conditions for this convergence were introduced in~\cite{graykieffer1980}. 
In particular, a process that is locally stationary will satisfy this (see \autoref{thm:ergthm}), even in the presence of a finite number of structural breaks.

\begin{proof}[Proof of \autoref{thm:consisgridlp}]
The result follows from \autoref{lem:intrandom} by observing that $(E,d_E)$ is a compact metric space and that the product measure $ \Delta_{1,n} \otimes \frac{1}{T}\sum_{t=1}^T \delta_{\mb{X}_{t,T}} \otimes \Delta_{b,n} \Rightarrow \lambda([0,1])\otimes \bar{\mu} \otimes \frac{1}{b}\lambda([0,b])$ weakly in probability as $n, T \to \infty$ in $\mc{P}(E)$, where $\lambda$ denotes the standard Lebesgue measure on the real line. 
\end{proof}

\begin{lemma}\label{lem:intrandom}
  Let  $(\zeta_T: T \ge 1)$ be a stochastic process defined on  $(\Omega,\F,\pr)$ taking values in $D(S,\rnum)$, where $S$ is a compact metric space. Suppose that  $
  \zeta_T \rightsquigarrow \zeta$, 
  where $\zeta$ is a Gaussian process with a.s. continuous sample paths. Furthermore, let $\{\nu_T\}$ be a sequence of random probability measures from $(\Omega, \F,\pr)$ to $(S,\mc{S})$ s.t. $\nu_T \Rightarrow \nu$ weakly where $\nu$ is a deterministic probability measure on $(S,\mc{S})$. Then for $p > 0$
  \[
  \int_S |\zeta_T(s)|^p d \nu_T(ds) \Rightarrow \int_S |\zeta(s)|^p d \nu(ds).
  \]
\end{lemma}
\begin{proof}[Proof of \autoref{lem:intrandom}]
   Let $\eta_T= \nu_T-\nu$. Then $\{\eta_T\} \subset \mc{M}(S)$, where $\mc{M}(S)$ is the space of signed measures on $(S,\mc{S})$ endowed with the topology of weak convergence. Note that the weak topology on this space is metrizable when restricted to any subset of measures with uniformly bounded total variation. Indeed, the metric
    \[
    d_{FM}(\mu,\nu)=\sup\{|\int fd\mu-\int fd\nu|: \|f\|_{\infty} \le 1, \text{Lip}(f)\le 1\},
    \]
    where $\text{Lip}(f)$ denotes the smallest constant $K$ such that for all $x,y \in S$, $|f(x)-f(y)|\le K |x-y|$, induces the weak topology on $\mc{M}_2(S)=\{ \mu \in \mc{M}(S): \|\mu\|_{TV} \le 2\}$. Note that this metric is equivalent to the Bounded Lipschitz metric introduced by Dudley \citep[see e.g.,][]{fortet1953,dudley2002}.  Now, since $\zeta_T \rightsquigarrow \zeta$ and $d_{FM}(\nu_T,\nu) \overset{p}{\to} 0$, we obtain joint weak convergence in $D(S) \times \mc{M}_2(S)$, i.e., 
    \[
    (\zeta_T, \eta_T) \rightsquigarrow (\zeta,0), \tageq \label{eq:joinwc}
    \]
    where it is understood that $\zeta \in C(S)$ a.s. and $0$ denotes the zero measure. 
     Define $\Phi: D(S) \times \mc{M}_2(S) \to \rnum$ by 
     \[
     \Phi(f,\varpi)= \int_S |f(x)|\varpi(dx).
     \]
     In order to apply the CMT, we  show that this functional is continuous at $(f,0)$ where $f \in C(S)$. To this end, let $f_T \to f$ in $D(S)$ and $\eta_T \Rightarrow 0$ weakly in $\mc{M}_2(S)$; then it suffices to show $\Phi(f_T, \eta_T) \to \Phi(f,0)=0$. By the triangle inequality
     \begin{align*}
     \bv \int_S|f_T(x)|\eta_T(dx)\bv  \le  \int_S|f_T(x)-f(x)| |\eta_T|(dx) + \bv \int_S |f(x)|\eta_T(dx) \bv 
     \end{align*}
     where $|\eta_T|=\eta_T^++\eta_T^{-}$. Since $f$ is continuous on $S$, 
$
\epsilon_T = \sup_{x\in S}|f_T(x)-f(x)| \to 0$,
and thus 
\[\int_S|f_T(x)-f(x)| |\eta_T|(dx) \le 2 \epsilon_T \to 0.\]
For the second term, we note that $f$ is  continuous on a  compact metric space and so $|f|$ is a bounded continuous function. Since $\eta_T \Rightarrow 0$ weakly, the second term therefore also goes to zero.  
Consequently, $\Phi$ is continuous at $(\zeta,0)$ a.s., and the continuous mapping theorem together with \eqref{eq:joinwc} yields
\[
\Phi(\zeta_T,\eta_T) \Rightarrow \Phi(\zeta,0).
\]
Note that this argument can be extended to $L^p$-based test statistics for any nonnegative $p \neq 1$ by an application of the continuous mapping theorem applied to the composition with the map $x \mapsto |x|^p$.
\end{proof}
    
Classical (multi-parameter) hypothesis testing based on discrepancy measures typically requires evaluating $U$-statistics, which demand proving that higher-order dependence structures are asymptotically negligible and accounting for degeneracy issues. It is worth emphasizing that the $L^p$-based statistic formulated in \autoref{thm:consisgridlp} sidesteps the need for a $U$-statistic representation by leveraging the direct functional convergence of the empirical process and the weak invariance principles for the random fields. \autoref{lem:intrandom} establishes that this simpler, direct formulation is fully sufficient to preserve weak convergence to a functional of a Gaussian process. 

Crucially, \autoref{thm:consisgridlp} decouples the practical implementation of the test from its theoretical validation: while in practice the statistic is evaluated exclusively over the discrete grid of available sample points via $\bar{\mu}_T$, the underlying proof relies on the multi-parameter weak invariance principles introduced above. Because this underlying limit theorem treats the metric space itself as a continuous parameter to establish tightness, a topological and geometric trade-off emerges: both the supremum and $L^p$-based approaches require the parameter space $(S,d_S)$ to be compact (or locally compact and Polish) to validate functional convergence, alongside explicit control on the geometric complexity. As demonstrated in the next section, employing a $U$-statistic framework (if possible) implicitly bypasses these constraints that are inherent to weak invariance principles. However, the underlying metric space must still possess sufficient topological structure to support a separating kernel for its valid use in the context of measure-theoretic inference.

 \subsection{Weak invariance principles under Hilbertian geometry}
 In this section, weak invariance principles are established that allow one to estimate the integral probability metric where the class $\Gamma$ is the unit ball of a RKHS, suitable if the metric space adheres to \autoref{thm:isombedSch}. Limit theorems are derived both in the degenerate and nondegenerate case, which require a different scaling parameter to ensure a nondegenerate limiting distribution \citep[see e.g.,][and references therein]{BBD01,Leucht2013,wendler2012,ginenickl2016}.

\subsection{Weak invariance principles for multi-parameter $U$-processes }\label{sec:weinpr}

If \autoref{thm:isombedSch} holds, then we can find a strictly positive definite continuous kernel such that \eqref{eq:gamballrkhs} defines a metric. Since
\begin{align*}
    \|m_{\mu}-m_{\nu}\|^2_{\rkhs}  &= \int_S \int_S k(x,y) \mu(dx) \mu(dy) + \int_S \int_S k(x,y) \nu(dx) \nu(dy)
    \\& -2 \int_S \int_S k(x,y) \mu(dx) \nu(dy), \tageq \label{eq:mmdpop}
\end{align*}
the problem of estimation reduces to the analysis of $U$-statistics. 
In order to estimate population quantities such as \eqref{eq:mmdpop} for a process that is potentially nonstationary, a flexible building block is the following partial sum process $\{\mathscr{S}_T(u,v,\lambda): u,v \in [0,1], \lambda \in [0,\Lambda]\}$ defined by 
\[
\mathscr{S}_T(u,v,\lambda)=   \frac{1}{T^{2}} \sum_{t=1}^{\flo{u T}}\sum_{s=1}^{\flo{vT}}h({\mb{X}}_{t,T}, {\mb{X}}_{s,T}, \lambda),\tageq\label{eq:bbpro}
\]
where $h: S \times S \times [0, \Lambda] \to \rnum$ is a kernel function that is symmetric in its first two arguments. The demeaned kernel $h$ will be denoted by $\hti$.
Under regularity conditions,
\[
\E \mathscr{S}_T(u,v,\lambda) \to 
\int_0^u \int_0^v \Big(\int_S\int_S  h(x,y,\lambda)\mu^X_\eta(dx)\mu_w^X(dy)\Big) d\eta d w.
\]
where $\mu_\eta^X = \pr\circ \mb{X}^{-1}_0(\eta)$ denotes the marginal distribution of $(\mb{X}_0(\eta): \eta \in [0,1])$ at rescaled time $\eta \in [0,1]$. 
An appropriate scaling factor for the process in \eqref{eq:bbpro} to obtain a nondegenerate limiting distribution depends heavily on the properties of the kernel $h$. We distinguish below between the nondegenerate and degenerate case, which will be relevant in relation to the respective null and alternative hypotheses for testing equality of measures, and calculation of corresponding power properties of such tests.

\subsubsection{Nondegenerate case}

\begin{thm} \label{thm:fcltuv}
Assume that $\{\mb{X}_{t,T}\}$ is a Polish-valued locally stationary process with auxiliary process $\{\mb{X}_t(u): t\in \znum, u \in [0,1]\}$ as defined in \autoref{def:Localstationarity} with representations  as in \autoref{as:rep}, and such that both satisfy \eqref{eq:delta0} with $\delta_0(m) \to 0$ as $m \to \infty$.
In addition, suppose that
\begin{enumerate}[label=(\roman*)]
\item  $\sup_{T \in \mb{N}}\sup_\lambda \sup_{k \ge 0} \sum_{j=0}^{\infty} \sup_{t} \nu^{W}_{t,t-k,T}(j,\lambda) <\infty$ for $W \in \{\mb{X}, \mb{X}(\cdot)\}$, where 
\begin{align*}
    \nu^{W}_{t,t-k,T}(j,\lambda)&= \sup_{m \ge 1} \|P_{t-j}\big(\hti(W_{m,t,T}, W_{m,t-k,T}, \lambda)\big)\|_{\rnum,2};
\end{align*}

\item 
$\sup_T \sup_{k \ge 0} \sum_{j=0}^\infty j\sup_t \Big\| P_{t-j}\big(\hti({\mb{X}}_{t,T}, {\mb{X}}_{t-k,T},\lambda)\big)\Big\|_{\rnum,6+\epsilon}<\infty;$
\item the kernel function $h: S \times S \times [0,\Lambda]$ is nonnegative, globally bounded, symmetric in its first two arguments, monotone in its third argument, and locally Lipschitz continuous in each component, that is, 
\begin{align*}
    |h(x,y,\lambda)- h(x^\prime, y, \lambda)| & \le L_1 d_S(x,x^\prime)
    \\  |h(x,y,\lambda)- h(x, y, \lambda^\prime)| & \le L_2(x,y) |\lambda-\lambda^\prime|,
\end{align*}
where the Lipschitz function $L_2$ satisfies $\sup_{T \in \mb{N}} \sup_{1 \le t,s \le T} \mathbb{E} L^2(\mb{X}_{t,T}, \mb{X}_{s,T}) < \infty$ and $\sup_{u, v \in [0,1]} \sup_{t,s} \mathbb{E} L^2(\mb{X}_t(u), \mb{X}_s(v)) < \infty$. 
\end{enumerate}
Then 
\[\Big\{T^{1/2} \Big(\mathscr{S}_T(u,v,\lambda)-\E \mathscr{S}_T(u,v,\lambda)\Big)\Big\}_{u,v \in [0,1], \lambda \in [0,\Lambda]}{\rightsquigarrow} \Big\{\mathbb{G}(u,v,\lambda)\Big\}_{u,v \in [0,1],\lambda \in [0,\Lambda]} \quad (T \to \infty) \]
in $D([0,1]^2\times [0,\Lambda])$ w.r.t. the Skorokhod topology,
 where $\{\mathbb{G}(u,v,\lambda): u,v \in [0,1], \lambda\in [0,\Lambda]\}$  is a zero-mean Gaussian process of which the covariance structure is given by
\begin{align*}
Cov(\mb{G}(u_1, v_1,\lambda_1), \mb{G}(u_2,v_2, \lambda_2))&= \int_0^{\min(u_1,u_2)} \sigma_\eta( v_1,\lambda_1, v_2,\lambda_2) d \eta
\\&+ \int_0^{\min(u_1,v_2)} \sigma_\eta( v_1,\lambda_1,u_2,\lambda_2) d \eta
\\&+ \int_0^{\min(v_1,u_2)} \sigma_\eta( u_1,\lambda_1,v_2,\lambda_2) d \eta
\\&+ \int_0^{\min(v_1,v_2)} \sigma_\eta( u_1,\lambda_1,u_2,\lambda_2) d \eta,
\end{align*}
where
\begin{align*}
&\sigma_\eta(z,\lambda_1,z^\prime,\lambda_2)
\\&= \sum_{k \in \znum} \text{Cov}\Big(\int_0^{z}\E_{{\mb{X}}^\prime_0(w)}[\hti({\mb{X}}_0(\eta), {\mb{X}}^\prime_0(w),\lambda_1)]dw, \int_0^{z^\prime}\E_{{\mb{X}}^\prime_0(w)}[\hti({\mb{X}}_k(\eta),{\mb{X}}^\prime_0(w),\lambda_2)]dw\Big),
\end{align*}
and where $(\mb{X}^\prime_t(u): t\in \znum, u \in [0,1])$ denotes an independent copy of  $(\mb{X}_t(u): t\in \znum, u \in [0,1])$. 
\end{thm}

Several remarks are in order. 
First, this result is much more general than is necessary to characterize IPM-based statistics under the alternative and does not rely on a positive-definite kernel. This theorem can also be used to construct tests for various other hypotheses such as independence, inter-point distance distributions, correlation dimensions etc. 
   Second, the condition in (ii) is used to control the regularity of the sample paths and ensure that the limit process has continuous sample paths almost surely. If,  however, one or more of these three parameters is treated as fixed in the analysis, i.e., if we fix $v, \lambda$ then the factor  $6+\epsilon$ in condition (ii) can be reduced to $2+\epsilon$, which can be compared with \autoref{thm:conditionsfortightpol}. If one is only interested in the finite-dimensional distributions then (ii) can be dropped altogether.
    Finally, in the iid case, the dependence conditions in (i) and (ii) are naturally satisfied. Sufficient conditions on the dependence measures on the original space for these type of dependence measures to hold were studied in \cite{vDB25}.

\begin{Proposition}\label{prop:infsumrep}
    We can represent the process as  
\[
\mathbb{G}(u,v,\lambda):=\sum_{\ell=1}^{\infty} \int_0^u \tilde{g}_\ell(\eta,v,\lambda) d\mathbb{B}_\ell(\eta)+ 
\sum_{\ell=1}^{\infty} \int_0^v \tilde{g}_\ell(\eta,u,\lambda) d\mathbb{B}_\ell(\eta),  \tageq\label{eq:GauslimUproc}\]
where $\{\mb{B}_{\ell}\}_{\ell \in \mb{N}}$ is a sequence of independent standard Brownian motions and where the sequence of functions $\{\tilde{g}_{\ell}\}$ is such that 
\[
\sigma_\eta(u_1,\lambda_1, u_2,\lambda_2)= \sum_{\ell=1}^\infty \tilde{g}_{\ell}(\eta, u_1,\lambda_1)\tilde{g}_{\ell}(\eta, u_2,\lambda_2). 
\]
\end{Proposition}

\begin{Corollary}
    In the stationary case, the previous representation reduces to
\[
\mb{G}(u,v,\lambda) = \sum_{\ell=1}^\infty \tilde{g}_\ell(v,\lambda) \mb{B}_\ell(u) + \sum_{\ell=1}^\infty \tilde{g}_\ell(u,\lambda) \mb{B}_\ell(v).
\]
\end{Corollary}

\begin{Remark}[Relation to prior work] \autoref{prop:infsumrep} establishes the full infinite-rank representation for the process $\{\mb{G}(u,v,\lambda): u,v \in [0,1], \lambda \in [0,\Lambda]\}$. Note that this generalizes the limiting representation of the process in \cite{vDB25}, provided the conditions of Theorem~4.7 therein are satisfied for that particular kernel. While the proofs establishing uniform convergence in that work remain entirely robust, the formal statement of the limiting process and its underlying covariance structure implicitly restricted the representation to a rank-1 approximation. The truncation error of this lower-dimensional representation is strictly governed by the spectral decay of the covariance operator. For sufficiently regular kernels—such as bounded analytic kernels—the self-normalized construction of the test statistic mathematically reduces this omitted tail remainder to a negligible threshold. While the methodology in \cite{vDB25} utilized a data-adaptively smoothed empirical cdf, this smoothing is anticipated to induce a similarly fast decay. This mechanism provides a strong theoretical justification for why the empirical results and simulated critical values obtained under the rank-1 approximation remained highly robust in those finite-sample simulations, though formalizing this spectral behavior for data-adaptive kernels is left for future work. \end{Remark}

\subsubsection{Degenerate case}\label{sec:degsection}

In the following, denote $K_x(\cdot)=k(x,\cdot) = \int_0^{\Lambda} h(x,\cdot,r)  \rho(dr)$ where $\rho$ is a nonnegative Borel measure on $\mc{B}([0,\Lambda])$ with $\text{supp}(\rho) \cap (0,\Lambda] \neq \emptyset$. 
The demeaned kernel will be denoted by   $\tilde{K}_{\mb{X}_{t,T}}= K_{\mb{X}_{t,T}}-m_{\mu_{t,T}}$ where $m_{\mu_{t,T}}(x) = \int_S k(x,y)\mu_{t,T}(dy)$.

\begin{thm} \label{thm:fcltuvdeg}
Assume that $\{\mb{X}_{t,T}\}$ is a Polish-valued locally stationary process with auxiliary process $\{\mb{X}_t(u): t\in \znum, u \in [0,1]\}$ as defined in \autoref{def:Localstationarity} with respective representations as in \autoref{as:rep}, and such that both satisfy \eqref{eq:delta0} with $\delta_0(m) \to 0$ as $m \to \infty$. In addition, suppose that 
\begin{enumerate}[label=(\roman*)]
\item The function $h: S \times S \times [0,\Lambda]$ is positive-definite and bounded for every fixed $r \in [0, \Lambda]$, and continuous in all arguments. 
\item  $\sup_T \sum_{j=0}^{\infty}\sup_{t \in [T]} \nu^{\mb{X}}_{t,T}(j) <\infty$ and $\sup_u \sum_{j=0}^{\infty}\sup_{t \in \znum} \nu^{\mb{X}(u)}_{t}(j) <\infty$, respectively, 
where
\begin{align}
\nu^{\mb{X}}_{t,T}(j) &= \sup_{m \ge 1}\|P_{t-j}({\tilde{K}}_{{\mb{X}_{m,t,T}}})\|_{\rkhs,2} \label{eq:nus_array} \\
\nu^{\mb{X}(u)}_{t}(j) &= \sup_{m \ge 1}\|P_{t-j}({\tilde{K}}_{\mb{X}_{m,t}(u)})\|_{\rkhs,2}. \label{eq:nus_stat}
\end{align}
\end{enumerate}
Then the Hilbert-valued process $\{\xi_{T}(u): u \in [0,1]\}$  defined by 
\[
\xi_{T}(u) :=\frac{1}{\sqrt{T}} \sum_{t=1}^{\flo{uT}} \tilde{K}_{\mb{X}_{t,T}}
\]
converges with respect to the Skorokhod topology on $D([0,1], \rkhs)$ to a zero-mean Gaussian process with independent increments and covariance structure  defined via the family of incremental covariance operators $\{\Gamma(u): u \in [0,1] \} \subset D([0,1], S_1(\rkhs))$
\begin{align*}
\Gamma(\eta)= \sum_{\ell \in \znum} \E[\tilde{K}_{\mb{X}_0(\eta)}\otimes \tilde{K}_{\mb{X}_\ell(\eta)}]. \tageq \label{eq:Gammcov}
\end{align*}
\end{thm}

\begin{Proposition}\label{prop:prespdcase}
    We may represent the limiting process as 
    \[
    \xi(u) = \int_0^u \sqrt{\Gamma(\eta)} d\mb{W}_c(\eta),
    \]
where $\mb{W}_c$ is a cylindrical Wiener process on $\rkhs$.
Furthermore, we have the following It{\^o} isometry $$\E\bignorm{\int_0^u \sqrt{\Gamma(\eta)} d \mb{W}_c(\eta)}^2_{\rkhs} = \int_0^u \|\sqrt{\Gamma(\eta)}\|_{S_2(\rkhs)}^2 d\eta.$$
\end{Proposition}

    \begin{lemma}
\label{lem:inprodfcltdeg}
Under the assumptions of \autoref{thm:fcltuvdeg},
    \begin{align*}
     &  \{ \inprod{\xi_{T}(u)}{\xi_{T}(v)}_{\rkhs}: u, v \in [0,1] \}\\& \rightsquigarrow \Big\{\biginprod{\sum_{j=1}^\infty \int_0^u \sqrt{\lambda_j(\eta)}\phi_j(\eta) d\mb{B}_j(\eta)}{\sum_{j=1}^\infty \int_0^v \sqrt{\lambda_j(\eta)}\phi_j(\eta)d\mb{B}_j(\eta)}_{\rkhs}: u, v \in [0,1] \Big\}
    \end{align*}    
    with respect to the Skorokhod topology on $D([0,1]^2,\rnum)$. 
\end{lemma}

Under stationarity, the limiting process reduces to the field:
    \begin{align*}
    \Big\{
    \sum_{j=1}^\infty\lambda_j \mb{B}_j(u)\mb{B}_j(v): u, v \in [0,1]\Big\}.
    \end{align*}
    Note that by taking $u=v=1$, the limit coincides with the one obtained in \cite{Leucht2013}.

\begin{proof}[Proof of \autoref{lem:inprodfcltdeg}]
  Define the map $\Phi: D([0,1],\rkhs) \to D([0,1]^2,\rnum)$ such that for a path $x \in D([0,1], \rkhs)$ 
  \[
  \Phi(x)(u,v) = \inprod{x(u)}{x(v)}_{\rkhs},  \quad (u,v) \in [0,1]^2.
  \]
  By assumption, we have that $\xi_T \Rightarrow \xi$ as $T \to \infty$ with respect to the Skorokhod topology on $D([0,1],\rkhs)$. Furthermore, $\xi$ has almost surely continuous sample paths, i.e., $\pr(\xi \in C([0,1],\rkhs))=1$.  By Skorokhod's representation theorem there exist on a suitable probability space random elements $\{\tilde{\xi}_T\}_{T \in \mb{N}}$, $\tilde{\xi}$ such that $\mc{L}(\xi_T) =\mc{L}(\tilde{\xi}_T) $ for all $T$ and $\mc{L}(\xi) =\mc{L}(\tilde{\xi})$ and $d_{J_1}(\tilde{\xi}_T, \tilde{\xi}) \to 0$ almost surely, where 
  \[
  d_{J_1}(x,y) = \inf_{\lambda \in \Lambda} \max\Big\{\sup_{u \in [0,1]} |\lambda(u) -u|, \sup_{u \in [0,1]}\|x(u)-y(\lambda(u)) \|_{\rkhs}
\Big\}  \]
  and where $\Lambda$ is the set of all strictly increasing continuous bijections from $[0,1]$ to itself. We recall 
  that $d_{J_1}(\tilde{\xi}_T, \tilde{\xi}) \to 0$ almost surely implies that $\sup_{u \in [0,1]}\|\tilde{\xi}_T(u)- \tilde{\xi}(u)\|_{\rkhs} \to 0$ almost surely. Standard calculations yield that for any $x,y \in D([0,1],\rkhs)$:
  \begin{align*}
      \|\Phi(x)-\Phi(y)\|_{\infty} &=\sup_{u,v} |\inprod{x(u)}{x(v)}_{\rkhs}-\inprod{y(u)}{y(v)}_{\rkhs}| 
      \\& \le (\sup_u\|x(u)\|_{\rkhs} +\sup_{u}\|y(u)\|_{\rkhs}) 
      \sup_{u}\|x(u)-y(u)\|_{\rkhs}.
  \end{align*}
  Hence, the map $\Phi$ is continuous with respect to the uniform topology on bounded sets of $D([0,1],\rkhs)$.
By the extreme value theorem $\sup_u\|\tilde{\xi}(u)\|_{\rkhs}$ is almost surely bounded 
and $\sup_u\|\tilde{\xi}_T(u)\|_{\rkhs}$ is almost surely bounded since $\sup_{u}\|\tilde{\xi}_T(u)-\tilde{\xi}(u)\|_{\rkhs} \to 0$ almost surely. 
 Consequently, the algebraic bound yields $\|\Phi(\tilde{\xi}_T)- \Phi(\tilde{\xi})\|_{\infty} \to 0$ almost surely and thus $\Phi(\tilde{\xi}_T) \rightarrow \Phi(\tilde{\xi})$ almost surely w.r.t. the uniform topology on $D([0,1]^2,\rnum)$. This in turn implies a.s. convergence w.r.t. the Skorokhod topology on $D([0,1]^2,\rnum)$. By separability of the latter, this establishes weak convergence and by the equality of laws, that  $\Phi({\xi}_T) \Rightarrow \Phi({\xi})$ with respect to the Skorokhod topology on $D([0,1]^2,\rnum)$.

 \end{proof}

\subsubsection{Localized discrepancy measures}\label{sec:app}

To illustrate the practical utility of the derived machinery, we highlight its applicability in the context of testing for discontinuities in measure. Crucially, our weak invariance principles for nonstationary random fields allow us to easily incorporate dynamic bandwidths to explicitly characterize the boundary behavior between the degenerate and nondegenerate regimes.

Suppose that \autoref{thm:isombedSch} holds, i.e., the Laplace kernel $k(x,y)=e^{-\lambda d_S(x,y)}$, for some $\lambda>0$, satisfies the following:

\begin{assumption}\label{as:conditionsonK}
    The class of functions $\{k(x,\cdot): x \in S\}$  satisfies the conditions of \autoref{prop:contemb} and $k$ is strictly positive definite.
\end{assumption}
 
To detect a local discontinuity in the marginal distribution at rescaled time $u \in (0,1)$, the stated assumption allows us to formulate a test based on the squared metric
\[
\gamma^2(\mu_{u^+},\mu_{u^-})= \bignorm{\int_S k(\cdot,y)\mu_{u^+}(dy)- \int_S k(\cdot,y)\mu_{u^-}(dy)}^2_{\rkhs} \tageq \label{eq:rkhsbreak}
\]
More specifically, we consider the following pair of hypotheses
\[
H_0 : \gamma^2(\mu_{u^+},\mu_{u^-})=0 \quad \text{ vs } \quad H_1: \gamma^2(\mu_{u^+},\mu_{u^-}) \neq 0.  
\]
A natural estimator takes the form of an appropriately scaled version of
\begin{align*}
\hat{\gamma}^2_T(u,b)= \frac{1}{\flo{bT}^{2}} \Big(&\sum_{t,s=\flo{uT}-\flo{bT}+1}^{\flo{u T}} h({\mb{X}}_{t,T}, {\mb{X}}_{s,T}, \lambda)+\sum_{t,s=\flo{u T}+1}^{\flo{uT}+\flo{bT}} h({\mb{X}}_{t,T}, {\mb{X}}_{s,T}, \lambda)\\& - 2\sum_{t=\flo{uT}-\flo{bT}+1}^{\flo{u T}}\sum_{s=\flo{u T}+1}^{\flo{uT}+\flo{bT}} h({\mb{X}}_{t,T}, {\mb{X}}_{s,T}, \lambda)\Big), \quad (u,b) \in Q^\star
\end{align*}
where $h(\cdot, \cdot, \lambda)=k(\cdot,\cdot)$ and where $Q^\star=\{(u,b) \in \rnum^2: 0\le u\le 1, 0 \le b \le \min(u,1-u)\}$. Recall that we use the convention that $\hat{\gamma}^2_T(u,0)=0$. Here, $b$ takes the role of a bandwidth parameter, and may be taken to converge to zero with the sample size such that $b T \to \infty$ (see \autoref{thm:limitbtdeg}-\autoref{thm:limitbtnondeg}). Note that as $b$ gets smaller, the test is concerned with detecting a discontinuity at point $u$.

Taking the class of functions to be the unit ball of the RKHS for the integral probability metric leads to the classical issue that the appropriate scaling for a nondegenerate limit depends heavily on whether the null hypothesis holds. We make this precise for the above test. Since we assume  \autoref{thm:isombedSch} and thus that the kernel satisfies \autoref{as:conditionsonK}, Parseval's identity shows that the estimator of $\gamma^2(\mu^+,\mu^-)$ can be written as
\begin{align*}
\frac{\flo{bT}^2}{T^2}\hat{\gamma}_T^2(u,b) = \sum_{i=1}^\infty \frac{1}{T^2}\Big( \sum_{t=\flo{uT}+1}^{\flo{uT}+\flo{bT}} \inprod{K_{\mb{X}_{t,T}}}{e_i} - \sum_{t=\flo{uT}-\flo{bT}+1}^{\flo{uT}} \inprod{K_{\mb{X}_{t,T}}}{e_i} \Big)^2,  
\end{align*}
where $\{e_i\}_{i \in \mb{N}}$ is an ONB of the $\rkhs$. 
Recall the notation from \autoref{sec:degsection}, and define
\begin{align*}
    \xi_{T}^+(u,b)(\cdot) &:=\frac{1}{\sqrt{T}} \sum_{t=\flo{uT}+1}^{\flo{uT}+\flo{bT}} \tilde{K}_{\mb{X}_{t,T}}( \cdot)
\\ \xi_{T}^-(u,b)(\cdot)& :=\frac{1}{\sqrt{T}} \sum_{t=\flo{uT}-\flo{bT}+1}^{\flo{uT}} \tilde{K}_{\mb{X}_{t,T}}( \cdot).
\end{align*}
Expanding it further, we obtain
\begin{align*}
\frac{\flo{bT}^2}{T^2}\hat{\gamma}_T^2(u,b) &
    = \sum_{i=1}^\infty \frac{1}{T}\Big( \inprod{\xi^{+}_{T}(u,b)}{e_i}_{\rkhs} -\inprod{\xi^{-}_{T}(u,b)}{e_i}_{\rkhs}  \Big)^2  
    \\& + \sum_{i=1}^\infty \frac{1}{T^2}\Big( \sum_{t=\flo{uT}+1}^{\flo{uT}+\flo{bT}} \inprod{m_{\mu_{t,T}}}{e_i}-\sum_{t=\flo{uT}-\flo{bT}+1}^{\flo{uT}}\inprod{m_{\mu_{t,T}}}{e_i}\Big)^2
    \\& +2 \sum_{i=1}^\infty \frac{1}{T^2}\Big(\sum_{t=\flo{uT}+1}^{\flo{uT}+\flo{bT}} \inprod{\tilde{K}_{\mb{X}_{t,T}}}{e_i} - \sum_{t=\flo{uT}-\flo{bT}+1}^{\flo{uT}} \inprod{\tilde{K}_{\mb{X}_{t,T}}}{e_i} \Big) 
    \\&\quad \quad \quad \times \Big(\sum_{t=\flo{uT}+1}^{\flo{uT}+\flo{bT}} \inprod{m_{\mu_{t,T}}}{e_i}-\sum_{t=\flo{uT}-\flo{bT}+1}^{\flo{uT}}\inprod{m_{\mu_{t,T}}}{e_i}\Big)~. \tageq \label{eq:expgammasq}
\end{align*}
If the process has time-invariant marginals within a ball around $u$, the second and third terms are zero.  Under this scenario, the estimator upscaled with $\sqrt{T}$ coincides with the $U$-statistic considered in \autoref{thm:weakconvnondeggamm}, and converges to  a degenerate limit. Indeed, the first term also converges to zero unless we upscale with a factor of $T$, in which case we obtain:
\begin{thm}\label{thm:nullconvub}
Under the conditions of \autoref{thm:fcltuvdeg}
\begin{align*} 
  &\Big\{\inprod{\xi^{+}_{T}(u,b)}{\xi^{+}_{T}(u,b)}_{\rkhs}+\inprod{\xi^{-}_{T}(u,b)}{\xi^{-}_{T}(u,b)}_{\rkhs}-2 \inprod{\xi^{+}_{T}(u,b)}{\xi^{-}_{T}(u,b)}_{\rkhs}: (u,b) \in Q^\star\Big\}
   \\&  \rightsquigarrow\Big\{ \|\xi^{+}(u,b)\|^2_{\rkhs}+\|\xi^{-}(u,b)\|^2_{\rkhs}-2 \inprod{\xi^{+}(u,b)}{\xi^{-}(u,b)}_{\rkhs}: (u,b) \in Q^\star\Big\},
\end{align*}
where $Q^\star=\{(u,b) \in \rnum^2: 0\le u\le 1, 0 \le b \le \min(u,1-u)\}$.
\end{thm}
Observe that $Q^\star$ is compact and that, for fixed $u,b$, the mean is given by  $\E \|\xi^{+}(u,b)\|^2_{\rkhs}+\E\|\xi^{-}(u,b)\|^2_{\rkhs} = \int_{u}^{u+b} \mathrm{tr}(\Gamma(\eta)) d\eta+\int_{u-b}^{u} \mathrm{tr}(\Gamma(\eta)) d\eta $
and the variance is given by
\[
2  \mathrm{tr}\Big(\big( \int_{u-b}^{u+b}\Gamma(\eta) \, d\eta \big)^2\Big)
\]
The limiting distribution corresponds to a localized version of the result in \cite[][Theorem 3.1]{Leucht2013}, though weaker assumptions are imposed here, and the method of proof does not rely on Mercer's theorem. 

\begin{proof}[Proof of \autoref{thm:nullconvub}]
    Define $\Phi_2: D([0,1], \rkhs) \to D([0,1]^2,\rkhs)$, $\Phi_2(x)(u,v)=x(v)-x(u)$, 
    as well as $\theta_+: D([0,1]^2, \rkhs) \to D(Q_1,\rkhs)$, $\theta_+(x)(u,v)=x(u,u+v)$ where $Q_1  = \{(u,v) \in \rnum^2: 0 \le u \le 1, -u \le v \le 1-u\}$. Then $\Delta=\theta_+ \circ \Phi_2 : D([0,1], \rkhs) \to D(Q_1, \rkhs)$, $\theta_+(\Phi_2(x))(u,v)=\Phi_2(x)(u,u+v)=x(u+v)-x(u)$. Consider  points $x,y \in D([0,1],\rkhs)$, then
    $\|\Delta(x)-\Delta(y)\|_{\rkhs,\infty}= \sup_{(s,t) \in Q_1}\|x(s+t)-x(s)-y(s+t)+y(s)\|_{\rkhs} \le 2\|x-y\|_{\rkhs,\infty} $, demonstrating that the mapping is globally Lipschitz with respect to the uniform topology on $D([0,1],\rkhs)$. In complete analogy, the same will be true if instead we consider $\theta_-: D([0,1]^2, \rkhs) \to D(Q_2,\rkhs)$, $\theta_-(\Phi_2(x))(u,v)=\Phi_2(x)(u-v,u)$ where $Q_2  = \{(u,v) \in \rnum^2: 0 \le u \le 1, u-1 \le v \le u\}$.  Since the limit has almost surely continuous sample paths, one may invoke the continuous mapping theorem to conclude directly, or alternatively, proceed along the lines of the argument in the proof of \autoref{lem:inprodfcltdeg} to establish the weak convergence of the process defined in the theorem   with respect to the Skorokhod topology on $D(Q^\star, \rnum)$.  
\end{proof}

To understand the non-degenerate limit that results under the alternative, we have the following closed form expression of the limit:
 
\begin{thm}\label{thm:weakconvnondeggamm}
 Under the conditions of \autoref{thm:fcltuv}, with the  $6+\epsilon$ condition in $(ii)$  relaxed to $4+\epsilon$, we have
\[ 
\Big\{\sqrt{T}b^2\Big(\hat{\gamma}^2_T(u,b)- \E \hat{\gamma}^2_T(u,b)\Big) : (u, b) \in Q^\star\Big\}\rightsquigarrow \{\mathbb{Z}(u,b): (u,b) \in Q^\star\}
\quad (T \to \infty). \]
where $\{\mathbb{Z}(u,b): (u,b) \in Q^\star\}$  is a zero-mean Gaussian process such that for each $(u,b) \in Q^\star$, the random variable $\mathbb{Z}(u,b)$ has variance given by 
\[
 4\Big(\int_{u}^{u+b}+\int_{u-b}^u\Big)\Big[\sigma^2_{b,+}(\eta, u) +\sigma^2_{b,-}(\eta, u)-2\sigma^2_{b,\pm}(\eta, u)\Big] d\eta \tageq \label{eq:covkernel}
\]
where  
\[
\sigma_{b,+}^2(\eta, u):=\sum_{\ell \in \znum} \text{Cov}\Big(\int_u^{u+b}\E_{{\mb{X}}^\prime_0(w)}[\hti(\mb{X}_0(\eta), \mb{X}^\prime_0(w))]dw, \int_u^{u+b}\E_{{\mb{X}}^\prime_0(w)}[\hti({\mb{X}}_\ell(\eta),{\mb{X}}^\prime_0(w))]dw\Big); 
\]
\[
\sigma_{b,-}^2(\eta, u):=\sum_{\ell \in \znum} \text{Cov}\Big(\int_{u-b}^{u}\E_{{\mb{X}}^\prime_0(w)}[\hti({\mb{X}}_0(\eta), {\mb{X}}^\prime_0(w))]dw, \int_{u-b}^{u}\E_{{\mb{X}}^\prime_0(w)}[\hti({\mb{X}}_\ell(\eta),{\mb{X}}^\prime_0(w))]dw\Big);
\]
\[
\sigma_{b,\pm}^2(\eta, u):=\sum_{\ell \in \znum} \text{Cov}\Big(\int_u^{u+b}\E_{{\mb{X}}^\prime_0(w)}[\hti({\mb{X}}_0(\eta), {\mb{X}}^\prime_0(w))]dw, \int_{u-b}^{u}\E_{{\mb{X}}^\prime_0(w)}[\hti({\mb{X}}_\ell(\eta),{\mb{X}}^\prime_0(w))]dw\Big). 
\]
\end{thm}
Note that under stationarity in the neighborhood of point $u$, the limiting distribution reduces to a degenerate random variable. The above is a building block in \cite{kvd26}, where the test can detect (relevant) as well as locate multiple abrupt changes with statistical guarantees in a time-varying process.

\begin{proof}[Proof of \autoref{thm:weakconvnondeggamm}]
Recall the notation from \autoref{thm:fcltuv} and denote 
\[
\Xi_T(u,v) = \sqrt{T}\big(S_T(u,v,\lambda)-\E S_T(u,v,\lambda)\big),
\]
so that $\{\Xi_T\}_{T \ge 1} \subset D([0,1]^2, \rnum)$. By \autoref{thm:fcltuv}, we know that $\Xi_T \Rightarrow \Xi$ w.r.t. the Skorokhod topology where $\Xi =\mb{G}(\cdot,\cdot,\lambda)$ and $\pr(\Xi \in C([0,1]^2,\rnum))=1$. Using techniques as in the proof of \autoref{thm:nullconvub}, we can define a map $\Upsilon: D([0,1]^2,\rnum) \to D( Q^\star,\rnum)$ that is globally Lipschitz continuous with respect to the uniform topology on $D([0,1]^2,\rnum)$ such that we can write
\[
\sqrt{T}b^2\big(\hat{\gamma}_T^2-\E \hat{\gamma}_T^2\big)= \Upsilon(\Xi_T).
\]
Hence,  by the continuous mapping theorem $\Upsilon(\Xi_T) \Rightarrow \Upsilon(\Xi)$, with respect to the Skorokhod topology on $D(Q^\star,\rnum)$ where the limiting process has covariance structure as given in the statement.
 
\end{proof}

A particularly important consequence of the weak invariance principles introduced is that we can naturally consider dynamic bandwidth parameters, which is of practical relevance as it establishes the asymptotic stability under dynamic bandwidths.

\begin{thm}\label{thm:limitbtdeg}
Suppose the conditions of \autoref{thm:fcltuvdeg} hold. Let $\{b_T\}$ be a sequence of bandwidth parameters such that $b_T \to 0$ and $b_T T \to \infty$ as $T \to \infty$. 
Then, 
\[
\Big\{ b_T^2 T\Big( \hat{\gamma}^2_T(u,b_T) - \E\hat{\gamma}^2_T(u,b_T) \Big): (u, b_T) \in Q^\star \cap (\rnum \times \{b_T\})\Big\} \Rightarrow 0 \quad (T \to \infty),
\]
where the weak convergence is with respect to the Skorokhod metric restricted to the intersection $Q^\star \cap (\rnum \times \{b_T\})$.
\end{thm}

\begin{thm}\label{thm:limitbtnondeg}
Suppose the conditions of \autoref{thm:weakconvnondeggamm} hold. Let $\{b_T\}$ be a sequence of bandwidth parameters such that $b_T \to 0$ and $b_T T \to \infty$ as $T \to \infty$. Then,
\[
\Big\{ b_T^2 \sqrt{T}\Big( \hat{\gamma}^2_T(u,b_T) - \E\hat{\gamma}^2_T(u,b_T) \Big): (u, b_T) \in Q^\star \cap (\rnum \times \{b_T\})\Big\} \Rightarrow 0 \quad (T \to \infty),
\]
where the weak convergence is with respect to the Skorokhod metric restricted to the intersection $Q^\star \cap (\rnum \times \{b_T\})$.
\end{thm}

\begin{proof}[Proof of \autoref{thm:limitbtdeg}/\autoref{thm:limitbtnondeg}]
   We first focus on \autoref{thm:limitbtnondeg}. Let $\hat{\mb{Z}}_T(u,b) = \sqrt{T}b^2 \big(\hat{\gamma}^2_T(u,b)- \E \hat{\gamma}^2_T(u,b)\big)$. By \autoref{thm:weakconvnondeggamm}, we have $\hat{\mb{Z}}_T \Rightarrow \mb{Z}$ with respect to the Skorokhod topology on $D(Q^\star,\rnum)$ and $\pr(\mb{Z} \in C(Q^\star,\rnum))=1$.  By Skorokhod's representation, there exists on an enriched probability space a sequence of random elements $\{\hat{\tilde{\mb{Z}}}_T\}_{T \ge 1} \subset D(Q^\star,\rnum)$ and $\tilde{\mb{Z}} \in D(Q^\star,\rnum)$ such that $\mc{L}(\hat{\mb{Z}}_T) =\mc{L}(\hat{\tilde{\mb{Z}}}_T)$ for all $T \ge 1$ and $\mc{L}(\mb{Z}) =\mc{L}(\tilde{\mb{Z}})$, and such that $\hat{\tilde{\mb{Z}}}_T \to \tilde{\mb{Z}}$ almost surely in $(D(Q^\star,\rnum),d_{J_1})$. Consequently, 
    \[
    \sup_{(u,b) \in Q^\star} |\hat{\tilde{\mb{Z}}}_T(u,b)-\tilde{\mb{Z}}(u,b)| \to 0 \quad \tageq \label{eq:remarkonbt}
    \]
    almost surely as $T \to \infty$. Observe that
   \begin{align*}
   &\sup_{\{u \in [0,1]: (u,b_T) \in Q^\star\}}|\hat{\tilde{\mb{Z}}}_T(u,b_T)-\tilde{\mb{Z}}(u,0)| 
   \\ & \le \sup_{\{u \in [0,1]: (u,b_T) \in Q^\star\}}
     \Big( |\hat{\tilde{\mb{Z}}}_T(u,b_T)-\tilde{\mb{Z}}(u,b_T)|+|\tilde{\mb{Z}}(u,b_T)-\tilde{\mb{Z}}(u,0)|\Big)
      \\& \le 
      \sup_{(u,b) \in Q^\star}|\hat{\tilde{\mb{Z}}}_T(u,b)-\tilde{\mb{Z}}(u,b)|+\sup_{\{u \in [0,1]: (u,b_T) \in Q^\star\}}|\tilde{\mb{Z}}(u,b_T)-\tilde{\mb{Z}}(u,0)|,
   \end{align*}
   which converges almost surely to 0 by \eqref{eq:remarkonbt} and the fact that $\pr(\tilde{\mb{Z}} \in C(Q^{\star},\rnum))=1$. More specifically, with probability 1, $\hat{\tilde{\mb{Z}}}_T(u,b_T) \to \tilde{\mb{Z}}(u,0)$ with respect to the uniform metric restricted to the intersection $Q^\star \cap (\rnum \times \{b_T\})$. This implies that  $\hat{\mb{Z}}_T(u,b_T) \Rightarrow  {\mb{Z}}(u,0)$ with respect to the Skorokhod metric restricted to the intersection $Q^\star \cap (\rnum \times \{b_T\})$. The result follows by noting that ${\mb{Z}}(u,0)=0$ almost surely. %
   A similar argument applies to the degenerate case with respect to the altered scaling factor and altered limit process. Details are omitted for the sake of brevity. 
\end{proof}

It is worth mentioning that for the particular hypothesis at hand, the more general setting introduced in \autoref{sec:general} can also be used.

\begin{lemma}\label{lem:apconsisgridlp}
   Assume that $\bar{\mu} = \lim_{T \to \infty} \bar{\mu}_T$ exists as an element of $\mc{P}(S)$. Then, under the conditions of \autoref{thm:conditionsfortightpol} or of \autoref{thm:conditionsfortight}, for a fixed bandwidth $b > 0$ such that $(u,b) \in Q^\star$ for all $u$ in the support of $\Delta_{1,n}$, we have
   \begin{align*}
&\frac{1}{a}\int_0^a \int_S \int_{\{u \in [0,1]: (u,b) \in Q^\star\}} \sqrt{T} \left| \int_S \tilde{f}_{y,r} (x) (\mu_{T,u+b}-\mu_{T,u})(dx) \right.
\\&\phantom{\frac{1}{a}\int_0^a \int_S\int \sqrt{T} \left| \right.} \left. - \int_S \tilde{f}_{y,r} (x) (\mu_{T,u}-\mu_{T,u-b})(dx) \right| \Delta_{1,n}(du) \, \bar{\mu}_T(dy) \, \Delta_{a,n}(dr)
   \\& \Rightarrow \frac{1}{a}\int_0^a \int_S \int_{\{u \in [0,1]: (u,b) \in Q^\star\}} \left| \zeta(u+b,y,r)+ \zeta(u-b,y,r)-2 \zeta(u,y,r) \right| du \, \bar{\mu}(dy) \, dr
  \end{align*}
  as $n \to \infty$ and $T \to \infty$, where we used the notation as in \autoref{thm:consisgridlp}.
\end{lemma}

We note again that a slight adaptation in the proof of \autoref{thm:limitbtnondeg}  will also allow us to show asymptotic stability under dynamic bandwidth regimes for this statistic,  with the minor change that we designate the interval of integration for parameter $r$ as $[0,a]$ to avoid notational clashes with the bandwidth parameter $b$. It is perhaps also interesting to note that if the kernel is strictly positive definite and $S$ is locally compact, then \autoref{cor:fcltmetric}, combined with an argument as in \autoref{thm:consisgridlp} and techniques from the proof of \autoref{thm:fcltuvdeg}, can be used to construct estimators of the left-hand side of \eqref{eq:gamballrkhs}. 
    If the kernel is not strictly positive definite, then the interpretation of \eqref{eq:gamballrkhs} as a metric breaks down. However, 
  \autoref{cor:fcltmetric} and/or \autoref{thm:consisgridlp} can be used for measure-theoretic inference as long as the ball property holds, where the $L^p$-based tests have the advantage over tests based on RKHS that they remain nondegenerate under both the null and alternative. This is an advantage over statistics where the limiting distribution is discontinuous at the boundary $\gamma=0$ of the parameter space, lacking uniform asymptotic validity \citep[see e.g.,][]{andrews2001,andrews2010,LeebPot05}. Indeed, such statistics are known to suffer from serious size distortions in finite samples and have poor power against local alternatives.  This issue will be aggravated by embedding the topological space into a Hilbert space if, in addition, the geometric property in \autoref{thm:isombedSch} fails. 
A detailed study of this, alongside the construction of formal statistical procedures
and their finite-sample evaluations, forms the subject of ongoing work.

\section*{Acknowledgements}
The author gratefully acknowledges support by NSF grant DMS-2311338.

\section*{Appendix}
Detailed mathematical proofs for the main theoretical results are provided in a technical appendix, which is available from the author upon request.


\begin{thebibliography}{}

\bibitem{andrews2001}
Andrews, D.~W.~K. (2001).
\newblock Testing when a parameter is on the boundary of the maintained hypothesis.
\newblock {\em Econometrica}, 69(3), 683--734.

\bibitem[{Andrews and Guggenberger(2010)}]{andrews2010}
Andrews, D.~W.~K. and Guggenberger, P. (2010).
\newblock Asymptotic size and a problem with subsampling and a normal or chi-square limit distribution.
\newblock {\em Econometric Theory}, 26(2), 426--468.

\bibitem{aronszajn1950}
Aronszajn, N. (1950).
\newblock Theory of reproducing kernels.
\newblock {\em Transactions of the American Mathematical Society}, 68(3), 337--404.

\bibitem{berbee1979}
Berbee, H.~C.~P. (1979).
\newblock {\em Random Walks with Stationary Increments and Renewal Theory}.
\newblock Mathematical Centre Tracts, Amsterdam.

\bibitem{berg1984}
Berg, C., Christensen, J.~P.~R., and Ressel, P. (1984).
\newblock {\em Harmonic Analysis on Semigroups: Theory of Positive Definite and Related Functions}.
\newblock Springer-Verlag, New York.

\bibitem[{Billingsley (1968)}]{bil68}
Billingsley, P. (1968).
\newblock {\em Convergence of Probability Measures}.
\newblock Wiley, New York.

\bibitem[{Billingsley (1995)}]{bil95}
Billingsley, P. (1995).
\newblock {\em Probability and Measure, 3rd Edition}.
\newblock John Wiley and Sons, New York.

\bibitem{bochner1932}
Bochner, S. (1932).
\newblock {\em Vorlesungen über Fouriersche Integrale}.
\newblock Akademische Verlagsgesellschaft, Leipzig.

\bibitem{borgwardt2006}
Borgwardt, K.~M., Gretton, A., Rasch, M.~J., Kriegel, H.-P., Sch{\"o}lkopf, B., and Smola, A.~J. (2006). 
\newblock Integrating structured biological data by kernel maximum mean discrepancy.
\newblock {\em Bioinformatics}, 22(14), e49--e57.

\bibitem{BBD01}
Borovkova, S., Burton, R., and Dehling, H. (2001). 
\newblock Limit theorem for functionals of mixing processes with applications to U-statistics and dimension estimation.
\newblock {\em Transactions of the American Mathematical Society}, 353(11), 4261--4318.

\bibitem[{Christensen (1980)}]{Christensen}
Christensen, J.~P.~R. (1980).
\newblock A survey of small ball theorems and problems.
\newblock {\em Lecture Notes in Mathematics}, 794, 24--30.

\bibitem{CramerWold1936}
Cram{\'e}r, H. and Wold, H. (1936).
\newblock Some theorems on distribution functions.
\newblock {\em Journal of the London Mathematical Society}, 11(4), 290--294.

\bibitem{dahlhaus1997}
Dahlhaus, R. (1997).
\newblock Fitting time series models to non-stationary processes.
\newblock {\em The Annals of Statistics}, 25, 1--37.

\bibitem{Davies}
Davies, R.~O. (1971).
\newblock Measures not approximable or specificable by means of balls.
\newblock {\em Mathematica}, 18, 157--160.

\bibitem{DZ08}
Davydov, Y. and Zitikis, R. (2008). 
\newblock On weak convergence of random fields.
\newblock {\em Annals of the Institute of Statistical Mathematics}, 60(2), 345--365.

\bibitem{Dudley1966}
Dudley, R.~M. (1966).
\newblock Convergence of Baire measures.
\newblock {\em Studia Mathematica}, 27(3), 251--268.

\bibitem{dudley2002}
Dudley, R.~M. (2002).
\newblock {\em Real Analysis and Probability}.
\newblock Cambridge University Press, Cambridge.

\bibitem{feller1971}
Feller, W. (1971).
\newblock {\em An Introduction to Probability Theory and Its Applications, Volume 2} (2nd ed.).
\newblock John Wiley \& Sons.

\bibitem{fortet1953}
Fortet, R. and Mourier, E. (1953).
\newblock Convergence de la loi de probabilit{\'e} de la moyenne d'un {\'e}chantillon.
\newblock {\em Studia Mathematica}, 13(1), 1--10.

\bibitem{gelfand1936}
Gelfand, I. (1936).
\newblock Sur les fonctions abstraitement continues et pures d'un certain type.
\newblock {\em Moscou Math{\'e}matique Recueil}, 4, 235--244.

\bibitem{ginenickl2016}
Gin{\'e}, E. and Nickl, R. (2016).
\newblock {\em Mathematical Foundations of Infinite-Dimensional Statistical Models}.
\newblock Cambridge University Press, Cambridge.

\bibitem{graykieffer1980}
Gray, R.~M. and Kieffer, J.~C. (1980).
\newblock Asymptotically mean stationary measures.
\newblock {\em The Annals of Probability}, 8(5), 962--973.

\bibitem{gretton2012}
Gretton, A., Borgwardt, K.~M., Rasch, M.~J., Sch{\"o}lkopf, B., and Smola, A.~J. (2012). 
\newblock A kernel two-sample test.
\newblock {\em Journal of Machine Learning Research}, 13(25), 723--773.

\bibitem{Gromov}
Gromov, M. (1999). 
\newblock {\em Metric structures for Riemannian and non-Riemannian spaces}.
\newblock Birkh{\"a}user, Basel.


\bibitem[Jordan and von Neumann(1935)]{JordanNeumann1935}
Jordan, P. and von Neumann, J. (1935).
\newblock On inner products in linear, metric spaces.
\newblock {\em Annals of Mathematics}, 36(3), 719--723.

\bibitem{kallenberg1997}
Kallenberg, O. (1997).
\newblock {\em Foundations of Modern Probability}.
\newblock Springer-Verlag, New York.


\bibitem{kvd26}
Knauth, W. and van Delft, A. (2026).
\newblock Detection and localization of discontinuities in measure within nonstationary time series of random geometric objects.
\newblock \emph{Manuscript in preparation}.



\bibitem{kolmogorov1933}
Kolmogorov, A.~N. (1933).
\newblock {\em Grundbegriffe der Wahrscheinlichkeitsrechnung}.
\newblock Springer, Berlin. (English translation: {\em Foundations of the Theory of Probability}, Chelsea, New York, 1950).

\bibitem{ledouxtalagrand}
Ledoux, M. and Talagrand, M. (1991).
\newblock {\em Probability in Banach Spaces: Isoperimetry and Processes}.
\newblock Ergebnisse der Mathematik und ihrer Grenzgebiete. Springer-Verlag, Berlin.

\bibitem[{Leeb and P{\"o}tscher(2005)}]{LeebPot05}
Leeb, H. and P{\"o}tscher, B.~M. (2005).
\newblock Model selection and inference: Facts and fiction.
\newblock {\em Econometric Theory}, 21(1), 21--59.

\bibitem{Leucht2013}
Leucht, A. and Neumann, M.~H. (2013).
\newblock Dependent wild bootstrap for degenerate $U$- and $V$-statistics.
\newblock {\em Journal of Multivariate Analysis}, 117, 257--280.

\bibitem{Levy1925}
L{\'e}vy, P. (1925).
\newblock {\em Calcul des probabilit{\'e}s}.
\newblock Gauthier-Villars, Paris.



\bibitem{Muller1997}
M{\"u}ller, A. (1997).
\newblock Integral probability metrics and their generating classes of functions.
\newblock {\em Advances in Applied Probability}, 29(2), 429--443.

\bibitem{parzen1962}
Parzen, E. (1962).
\newblock Probability density functionals and reproducing kernel Hilbert spaces.
\newblock In {\em Proceedings of the Symposium on Time Series Analysis}, pages 155--169. Wiley, New York.

\bibitem{parzen1963}
Parzen, E. (1963).
\newblock Probability Density Functionals and Reproducing Kernel Hilbert Spaces.
\newblock In Rosenblatt, M. (Ed.), {\em Proceedings of the Symposium on Time Series Analysis}, pages 155--169. Wiley, New York.

\bibitem{pettis1938}
Pettis, B.~J. (1938).
\newblock On integration in vector spaces.
\newblock {\em Transactions of the American Mathematical Society}, 44(2), 277--304.

\bibitem{PreissTisier91}
Preiss, D. and Ti{\v{s}}er, J. (1991).
\newblock Measures in Banach spaces are determined by their values on balls.
\newblock {\em Mathematika}, 38(2), 391--397.

\bibitem{Prokhorov1956}
Prokhorov, Y.~V. (1956).
\newblock Convergence of random processes and limit theorems in probability theory.
\newblock {\em Theory of Probability \& Its Applications}, 1(2), 157--214.

\bibitem{rosenblatt1959}
Rosenblatt, M. (1959).
\newblock Statistical analysis of stochastic processes with stationary residuals.
\newblock In {\em Probability and Statistics: The Harald Cram{\'e}r Volume}, pages 246--275. Almqvist \& Wiksell, Stockholm.

\bibitem{sanwagpan26}
Santoro, L.~V., Waghmare, K.~G., and Panaretos, V.~M. (2026).
\newblock Kernel embeddings and the separation of measure phenomenon.
\newblock {\em Proceedings of the National Academy of Sciences}, 123(23), e2522504123.

\bibitem{schoenberg1938}
Schoenberg, I.~J. (1938).
\newblock Metric spaces and positive definite functions.
\newblock {\em Transactions of the American Mathematical Society}, 44(3), 522--536.

\bibitem{schoenberg1938monotone}
Schoenberg, I.~J. (1938).
\newblock Metric spaces and completely monotone functions.
\newblock {\em Annals of Mathematics}, 39(4), 811--841.

\bibitem[Schwartz(1966)]{schwartz1966}
Schwartz, L. (1966).
\newblock {\em Th{\'e}orie des distributions}, volume~2.
\newblock Hermann, Paris.

\bibitem{simongab2023}
Simon-Gabriel, C.-J., Barp, A., Sch{\"o}lkopf, B., and Mackey, L. (2023).
\newblock Metrizing weak convergence with maximum mean discrepancies.
\newblock {\em Journal of Machine Learning Research}, 24(184), 1--20.

\bibitem{songmuller26}
Song, W. and M{\"u}ller, H.-G. (2026).
\newblock Inference for dispersion and curvature of random objects.
\newblock {\em Journal of the American Statistical Association}, 121(574), 729--740.

\bibitem{sriperumbudur2010}
Sriperumbudur, B.~K., Gretton, A., Fukumizu, K., Sch{\"o}lkopf, B., and Lanckriet, G.~R.~G. (2010). 
\newblock Hilbert space embeddings and metrics on probability measures.
\newblock {\em Journal of Machine Learning Research}, 11, 1517--1561.

\bibitem{SR06}
Subba~Rao, S. (2006).
\newblock On some nonstationary, nonlinear random processes and their stationary approximations.
\newblock {\em Advances in Applied Probability}, 38(4), 1155--1172.

\bibitem{Szekely1985}
Sz{\'e}kely, G.~J. (1985).
\newblock Homogeneity tests.
\newblock {\em Statistics \& Probability Letters}, 3(5), 263--265.

\bibitem{vakhania1987}
Vakhania, N.~N., Tarieladze, V.~I., and Chobanyan, S.~A. (1987).
\newblock {\em Probability Distributions on Banach Spaces}.
\newblock Mathematics and its Applications, vol. 14. D. Reidel Publishing Company, Dordrecht.

\bibitem[{van Delft and Blumberg (2025)}]{vDB25}
van Delft, A. and Blumberg, A.~J. (2025).
\newblock A statistical framework for analyzing shape in a time series of random geometric objects.
\newblock {\em Annals of Statistics}, 53(2), 561--588.

\bibitem{wendler2012}
Wendler, M. (2012).
\newblock $U$-processes under weak dependence.
\newblock {\em Stochastic Processes and their Applications}, 122(6), 2307--2326.

\bibitem{Wiener1958}
Wiener, N. (1958).
\newblock {\em Nonlinear Problems in Random Theory}.
\newblock Technology Press of MIT and John Wiley \& Sons, New York.

\bibitem{Wu2005}
Wu, W.~B. (2005).
\newblock Nonlinear system theory: Another look at dependence.
\newblock {\em Proceedings of the National Academy of Sciences}, 102(40), 14150--14154.

\end{thebibliography}
\end{document}